\documentclass{article}

\usepackage[utf8]{inputenc}
\usepackage[english]{babel}

\usepackage[a4paper]{geometry}
\usepackage{parskip}
\usepackage{amsmath, amssymb}
\usepackage{color}
\usepackage{graphicx}
\usepackage{upgreek} 
\usepackage{faktor}
\usepackage{xfrac}
\usepackage{commath}
\usepackage{amsthm}
\usepackage{mathtools}
\usepackage{stmaryrd} 
\usepackage{enumitem} 
\usepackage{verbatim} 
\usepackage{tikz-cd}
\usetikzlibrary{cd}

\usepackage{csquotes}
\usepackage[backend=biber,style=alphabetic,maxnames=6,sorting=nyt]{biblatex}

\definecolor{linkblue}{RGB}{1,1,190}
\definecolor{citegreen}{RGB}{1,190,1}
\usepackage[linkcolor=linkblue
           ,citecolor=citegreen
           ,ocgcolorlinks        
           ,bookmarksopen=True,  
           ]{hyperref}
\usepackage[ocgcolorlinks]{ocgx2} 

\usepackage{titling}
\predate{}
\postdate{}

\theoremstyle{definition}
\newtheorem{theorem}{Theorem}[section]
\newtheorem{definition}[theorem]{Definition}
\newtheorem{lemma}[theorem]{Lemma}
\newtheorem*{lemma*}{Lemma}

\newtheorem{proposition}[theorem]{Proposition}
\newtheorem{corollary}[theorem]{Corollary}

\newtheorem{remark}[theorem]{Remark}

\newcommand{\thistheoremname}{}
\newtheorem*{genericthm}{\thistheoremname}

\newcommand{\Q}{\mathbb{Q}}
\newcommand{\Z}{\mathbb{Z}}

\newcommand{\N}{\mathbb{N}}

\newcommand{\roi}{\mathcal{O}}

\DeclareMathOperator{\nrd}{nrd}
\DeclareMathOperator{\api}{\boldsymbol{\uppi}}
\DeclareMathOperator{\tensor}{\otimes}

\newcommand{\fp}{\mathfrak p}

\makeatletter
\newcommand*\bigcdot{\mathpalette\bigcdot@{.5}}
\newcommand*\bigcdot@[2]{\mathbin{\vcenter{\hbox{\scalebox{#2}{$\m@th#1\bullet$}}}}}
\makeatother

\newcommand\blfootnote[1]{%
  \begingroup
  \renewcommand\thefootnote{}\footnote{#1}%
  \addtocounter{footnote}{-1}%
  \endgroup
}

\tikzset{
    invisible/.style={opacity=0},
    labl/.style={anchor=south, rotate=90, inner sep=.5mm},
    symbol/.style={
    draw=none,
    every to/.append style={
      edge node={node [sloped, allow upside down, auto=false]{$#1$}}}
  }
}

\title{Elasticities of Orders in Central Simple Algebras}
\author{Casper Barendrecht}
\date{}

\begin{document}
\maketitle
\blfootnote{Date: \today}
\blfootnote{Email: \href{mailto:casper.barendrecht@uni-graz.at}{casper.barendrecht@uni-graz.at}}
\blfootnote{University of Graz, Institute for Mathematics and Scientific Computing, NAWI Graz, Heinrichstrasse 36, 8010 Graz, Austria}

\begin{abstract}
  Let $\roi$ be an order in a central simple algebra $A$ over a number field.
  The elasticitity $\rho(\roi)$ is the supremum of all fractions $k/l$ such that there exists an non-zero-divisor $a \in \roi$ that has factorizations into atoms (irreducible elements) of length $k$ and $l$.
  We characterize the finiteness of the elasticity for Hermite orders $\roi$, if either $\roi$ is a quaternion order, or $\roi$ is an order in an central simple algebra of larger dimension and $\roi_\fp$ is a tiled order at every finite place $\fp$ at which $A_\fp$ is not a division ring.
  We also prove a transfer result for such orders.
  This extends previous results for hereditary orders to a non-hereditary setting.
\end{abstract}

\section{Introduction}

If $\roi$ is a noetherian ring, then every non-zero-divisor $a \in \roi^\bullet$ can be written as a product of atoms (irreducible elements) of the monoid $\roi^\bullet$ of non-zero-divisors of $R$.
Usually such a factorization is not unique, and arithmetical invariants are used to quantify the non-uniqueness.
The \emph{elasticity} of an element $a \in \roi^\bullet$, denoted by $\rho(a) \in \mathbb Q_{\ge 1} \cup \{ \infty \}$, is the supremum of all $k/l$ such that $a$ has a factorization of length $k$ and a factorization of length $l$.
The \emph{elasticity} of $\roi^\bullet$ is then $\rho(\roi^\bullet) = \sup\{\, \rho(a) : a \in \roi^\bullet \,\}$.
The monoid $\roi^\bullet$ is \emph{half-factorial} if $\rho(\roi^\bullet) = 1$, that is, if the factorization length of every element is unique.
Elasticities are one of the most classical arithmetical invariants.

In factorization theory, one aims to study the various phenomena of the non-uniqueness of factorizations of elements using suitable arithmetical invariants, and in particular, to understand the interaction between arithmetical invariants and the underlying algebraic structure of the ring (see \cite{Anderson,Chapman,Fontana-Houston-Lucas,Geroldinger09,geroldinger16,smertnig16,armf,geroldinger-zhong20,mitft} for recent surveys and conference proceedings).
Factorization theory originated from algebraic number theory and has grown to be a subfield in the intersection of algebraic, analytic, and combinatorial number theory.
Whereas the interest has initially been focused on commutative settings, in particular, on Krull domains, Krull monoids, and Mori monoids, the past decade has seen an increasing interest in studying the factorizations in noncommutative rings and monoids.
While elasticities are some of the most basic invariants, they remain to be of central interest \cite{autry-oneill-ponomarenko20,gotti20,gotti-oneill20}.

The natural noncommutative analogue to orders in number fields are orders in central simple algebras over number fields.
Let $K$ be a number field and $R_S$ a ring of $S$-integers.
Let $A$ be a central simple algebra over $K$.
A subring $R_S \subseteq \roi \subseteq A$ is an \emph{$R_S$-order} in $A$ if it is finitely generated as $R_S$-module and contains a $K$-basis of $A$.
Factorizations in maximal $R_S$-orders, and more generally, hereditary $R_S$-orders have been studied before \cite{estes1991,smertnig2013,smertnig19}. 
Of particular interest are factorizations in quaternion orders and the associated metacommutation phenomenon \cite{MR3314061,MR3544510,MR3805465,baeth2017,chari2020,charibabei}.
Here we study \emph{non-hereditary} orders, and focus on their elasticities.

A \emph{transfer homomorphism} $\theta\colon R^\bullet \to D$ (see Definition~\ref{def:transferhom} below) preserves many arithmetical invariants.
In particular, $\rho(R^\bullet) = \rho(D)$.
If $\roi$ is a hereditary $R_S$-order that is moreover Hermite (see Definition~\ref{def:hermite} below), then there exists a transfer homomorphism with $D$ a monoid of zero-sum sequences over a (finite, abelian) ray class group of $K$.
Using adelic methods, our first main result, Theorem~\ref{prop:transfer}, shows that in the non-hereditary setting there still exists a transfer homomorphism, if one allows $D$ to be a $T$-block monoid (see Definition~\ref{def:Tblockmonoid}).
The $T$-block monoid $D$ is defined in terms of a ray class group of $K$ and the non-hereditary completions of $\roi$.
In the hereditary case, one recovers the known results.

We briefly discuss the crucial Hermite condition.
If $\dim_K A > 4$ or $A$ is an indefinite quaternion algebra, then $\roi$ is Hermite by strong approximation.
Thus the only non-Hermite orders appear in definite quaternion algebras.
Restrict now to the case where $R=R_S$ is the ring of algebraic integers of the center of $K$.
Then there exists a complete classification of the definite Hermite quaternion orders \cite{smertnigvoight2019}.
If $\roi$ is a maximal $R$-order that is not Hermite, it is known that there cannot be a transfer homomorphism to a monoid of zero-sum sequences \cite[Theorem 1.2]{smertnig2013}.
Thus, the Hermite condition is a natural one for a transfer result to hold.

Let $\roi_\fp$ denote the completion of $\roi$ at the prime $\fp$ of $R_S$.
If $A$ is a quaternion algebra, that is $\dim_K A = n^2 = 4$, then $\rho(\roi_\fp^\bullet)$ is finite if and only if $A_\fp$ is a division ring or $\roi_\fp$ is hereditary \cite{baeth2017}.
In case $n > 2$, we show that this continues to hold for tiled orders $\roi_\fp$ in Corollary~\ref{cor:tiledlocalelas}.
In our second main result, Theorem~\ref{thm:elasequivalence}, we characterize Hermite $R_S$-orders $\roi$ when $n=2$ or $\roi$ is tiled at every place that is not totally ramified.

\textbf{Acknowledgements.} The author acknowledges the support of the Austrian Science Fund (FWF): W1230.

I would like to thank my PhD supervisor Daniel Smertnig, for introducing me to this intriguing branch of mathematics and for his support and guidance over the past year.

\section{Preliminaries}

\subsection{Transfer Homomorphisms}
Throughout this paper a \textit{monoid} refers to a semigroup with identity, that is a triple \((T,\cdot,1)\) where \(T\) is a set, \(\cdot:T\times T\to T\) an associative function, and \(1\in T\) a neutral element with respect to \(\cdot\).
Unless otherwise stated, any monoid considered here is written multiplicatively.
An element \(a\in T\) is said to be \textit{cancellative} if, for all \(b,c\in T\) we have \(ab=ac\) implies \(b=c\), and \(ba=ca\) implies \(b=c\).
The set of cancellative elements is a submonoid of \(T\) and is denoted by \(T^\bullet\).
An element \(a\in T\) is invertible if there exists a \(b\in T\) such that \(ab=ba=1\).
The group of invertible elements is a submonoid of \(T^\bullet\) and is denoted by \(T^\times\).

Common examples of monoids studied in factorization theory include the monoid of cancellative elements of the multiplicative monoid of a ring \(R\), and \(R^\bullet\) will refer to this monoid.
Let \(\mathcal{P}\) be a set.
The \textit{free abelian monoid} \((\mathcal{F}(\mathcal{P}),\bigcdot)\) over \(\mathcal{P}\) is the set of all formal sequences \(S=g_1\bigcdot g_2\bigcdot ... \bigcdot g_{n_S}\) of elements of \(\mathcal{P}\), subject to the condition that \(g\bigcdot h=h\bigcdot g\), for all \(h,g\in \mathcal{P}\).
This set becomes a monoid under concatenation.
The number of elements in a sequence \(S\) is called the \textit{length} of \(S\) and we write \( |S|\) for this integer.
Any \(S\in \mathcal{F}(\mathcal{P})\) may be written as a formal product \(S=\prod_{g\in \mathcal{P}}^\bullet g^{n_g}\) with \(n_g\in \Z_{\geq 0}\), and \(n_g=0\) for all but finitely many \(g\).
Let \((G,\cdot)\) be a finite abelian group written multiplicatively.
Then the free abelian monoid \(\mathcal{F}(G)\) comes equipped with a natural homomorphism \(\sigma:
\mathcal{F}(G)\to G\) defined by \( \prod_{g\in G}^\bullet g^{n_g}\mapsto \prod_{g\in G}g^{n_g}\).
The preimage of the trivial element of \(G\) under this map is the \textit{monoid of zero-sum sequences over \(G\)} and is denoted by \(\mathcal{B}(G)\).
A sequence \(S=\prod_{g\in G}^\bullet g^{n_g}\in \mathcal{F}(G)\) is \textit{zero-sumfree} if, and only if, \(\sigma(S')\neq 1\) for all non-empty subsequences \(S'\) of \(S\).
Since \(G\) is a finite, there are only finitely many zero-sumfree sequences.
We define the \textit{Davenport constant} of \(G\) as \[\mathcal{D}(G)=1+\max \{|S|\mid S\in \mathcal{F}(G)\text{ is zero-sumfree}\}.\]
Equivalently, \(\mathcal{D}(G)\) may be defined as the maximum of the lengths of atoms of \(\mathcal{B}(G)\).
We refer to \cite{geroldinger2006} and \cite{smertnig16} for more background on this topic.
An immediate generalization of the monoid of zero-sum sequences is the notion of a \(T\)-block monoid as seen in \cite{geroldinger2006}.

\begin{definition}\label{def:Tblockmonoid}
Let \(T\) be a cancellative monoid and let \(\iota: T\to G\) be a monoid homorphism to a finite abelian group \((G,\cdot)\).
Let \(\sigma:\mathcal{F}(G)\to G\) be the natural homomorphism. 
The \textit{\(T\)-Block monoid over \(G\) induced by \(\iota\)} is the monoid
\[\mathcal{B}_T(G,\iota)=\{(S,\alpha)\in \mathcal{F}(G)\times T\mid \sigma(S)\cdot\iota(\alpha)=1\}.\]
\end{definition}
It is the pullback of \(\mathcal{F}(G)\) and \(T\) along \(\sigma\) and \(\iota':T\to G,\alpha\mapsto\iota(\alpha)^{-1}\), that is
\[\begin{tikzcd}
 \mathcal{B}_T(G,\iota)\arrow[r]\arrow[d] & T\arrow[d,"\iota' " ']\\
\mathcal{F}(G)\arrow[r,"\sigma"] & G.
\end{tikzcd}\]
If \(T=\{1\}\) is the trivial monoid, then \(\mathcal{B}_T(G,\iota)\) is simply the monoid of zero-sum sequences \(\mathcal{B}(G)\).

Let \(T\) be a cancellative monoid.
An element \(a\in T\) is an \textit{atom} if for every \(b,c\in T\) such that \(a=bc\), we have \(b\in T^\times\) or \(c\in T^\times\).
The set of atoms of \(T\) is denoted by \(\mathcal{A}(T)\).
An element \(a\in T\) is said to have a \textit{factorization} if there exist atoms \(u_1,...,u_k\in \mathcal{A}(T)\), and an \(\epsilon\in T^\times\) such that \(a=\epsilon u_1u_2\cdots u_k\) for some \(k\in \Z_{\geq 0}\).
The integer \(k\) is called the \textit{length} of the factorization \(\epsilon u_1u_2\cdots u_k\), and
the monoid \(T\) is \textit{atomic} if every \(a\in T\) admits a factorization.
A factorization of an element need not be unique in general, and in fact, two factorizations of the same element may have different lengths.
We recall some important definitions as seen in \cite{smertnig16}.

\begin{definition}
Let \(T\) be a cancellative monoid.

\begin{enumerate}
	\item The \textit{set of lengths} \(L(a)\) of an element \(a\in T\) is the set of all \(k\in \Z\) such that \(a\) has a factorization of length \(k\).
	\item The \textit{system of set of lengths of} \(T\) is the set \(\mathcal{L}(T)=\{L(a)\mid a\in T\}\).
\end{enumerate}

If \(T\) is atomic then:
\begin{enumerate}[resume]
\item The \textit{elasticity \(\rho(a)\)} of an element \(a\in T\backslash T^\times\) is given by
\[\rho(a)=\frac{\sup L(a)}{\min L(a)}\in \Q_{\geq 1}\cup\{\infty\},\]
and \(\rho(a)=1\) if \(a\in T^\times\).
\item The \textit{elasticity of \(T\)} is given by \(\rho(T)=\sup\{\rho(a)\mid a\in T\}\).
\item The \textit{refined elasticities of \(T\)} are given by 
\[\rho_k(T)=\sup\{\sup L(a)\mid a\in T\text{ with } \min L(a)\leq k\}\]
for all \(k\in \Z_{\geq 2}\).
\item For any \(k\in \Z_{\geq 2}\), the \textit{union of sets of lengths containing} \(k\) is given by
\[\mathcal{U}_k(T)=\{m\in\Z\mid\exists a\in T,\text{ with } \{k,m\}\subset L(a)\}.\]
\end{enumerate} 
\end{definition}

An immediate consequence of these definitions is the following.

\begin{lemma}\label{lem:elaslimit}
Let \(T\) be an atomic monoid.
The refined elasticities form an increasing sequence and \[\rho(T)=\limsup_{k\to\infty} \frac{\rho_k(T)}{k}.\]
In particular if \(\rho_k(T)=\infty\) for some \(k\in \Z_{\geq 2}\), then \(\rho(T)=\infty\).
\end{lemma}

The refined elasticities of an atomic monoid \(T\) are fully determined by its system of sets of lengths.
Thus if a homomorphism \(\varphi:T\to D\) between atomic monoids preserves the system of sets of lengths, then their elasticities coincide.
This leads to the definition of a transfer homomorphism.

\begin{definition}[{\cite[Definition 5.12]{smertnig16}}]\label{def:transferhom}
Let \(T,D\) be two atomic monoids.
A monoid homomorphism \(\varphi:T\to D\) is a \textit{transfer homomorphism} if the following conditions hold:
\begin{enumerate}
\item \(\varphi^{-1}(D^\times)=T^\times\), and \(D=D^\times \varphi(T) D^\times\),
\item If \(\varphi(a)=b_1b_2\), then there exist \(a_1,a_2\in T\), and an \(\epsilon\in D^\times\), such that \(a=a_1a_2\), \(\varphi(a_1)=b_1\epsilon\), and \(\varphi(a_2)=\epsilon^{-1}b_2\).
\end{enumerate}
\end{definition}

\begin{proposition}[{\cite[Theorem 5.15]{smertnig16}}]\label{theorem:transferhom}
Let \(T,D\) be atomic monoids, and let \(\varphi:T\to D\) be a transfer homomorphism.
Then \(L_T(a)=L_D(\varphi(a))\) for all \(a\in T\).
In particular \(\mathcal{L}(T)=\mathcal{L}(D)\).
\end{proposition}

A well-known application of a transfer homomorphism is \cite[Theorem 3.7.1]{geroldinger2006} which, given a maximal order \(R\) in a number field, provides a transfer homomorphism \(\varphi:R^\bullet\to \mathcal{B}(\text{Cl}(R))\) from the cancellative elements of \(R\) to the monoid of zero-sum sequences over its class group.
Let \(I\) be a set and let \((T_i)_{i\in I}\) be a family of cancellative monoids.
The \textit{restricted product} of \((T_i)_{i\in I}\) is the monoid
\[\widehat{T}=\sideset{}{'}\prod_{i\in I} T_{i}= \left\{(\alpha_{i})_{i}\in \prod_{i\in I}T_i~\middle| ~\alpha_{i}\in T_{i}^\times\text{ for all but finitely many }i \right\}.\]
  
\begin{lemma}\label{lem:restricttransfer}
Let \(I\) be a set.
Let \((T_i)_{i\in I}, (D_i)_{i\in I}\) be two families of atomic monoids, and let \((\varphi_i)_{i\in I}\) be a family of transfer homomorphisms \(\varphi_i:T_i\to D_i\)
The following statements hold.
\begin{enumerate}
\item The monoid \(\widehat{T}\) is atomic.
\item The map \(\widehat{\varphi}:\widehat{T}\to \widehat{D},(\alpha_i)_{i\in I}\mapsto (\varphi_i(\alpha_i))_{i\in I}\) is a transfer homomorphism.
\end{enumerate}
\end{lemma}

\begin{proof} Observe that \(\widehat{T}^\times =\prod_{i\in I}T_i^\times\).
Let \(\widehat{\alpha}=(\alpha_i)_{i\in I}\in \widehat{T}\), and assume that \(\widehat{\alpha}=\widehat{\beta}\cdot \widehat{\gamma}\) for certain \(\widehat{\beta},\widehat{\gamma}\in \widehat{T}\).
By definition of the restricted product, there exists a finite set \(J\subset I\) such that \(\alpha_i\in T_i^\times\) for all \(i\notin J\).
Consequently we have \(\beta_i,\gamma_i\in T_i^\times\) for all \(i\notin J\), and it follows that \(\widehat{\alpha}\) has a factorization in \(\widehat{T}\) if and only if \((\alpha_i)_{i\in J}\) has a factorization in \(\prod_{i\in J}T_i\).
Hence to prove the atomicity of \(\widehat{T}\) we may assume \(I\) to be finite.

Similarly as \(\varphi_i\) is a transfer homomorphism for all \(i\in I\), it follows that \(\widehat{\varphi}(\widehat{\alpha})_i=\varphi_i(\alpha_i)\) is invertible if \(i\notin J\).
By the same argument as before, in order to prove that \(\widehat{\varphi}\) is a transfer homomorphism, we may assume \(I\) to be finite.
By induction we may moreover assume that \(I=\{1,2\}\).

An element \(\alpha=(\alpha_1,\alpha_2)\in T_1\times T_2\) can be written as a product \((\alpha_1,1)\cdot (1,\alpha_2)\).
Hence any factorizations of \(\alpha_1\) and \(\alpha_2\) lift to a factorization of \(\alpha\), proving the atomicity of \(\widehat{T}\).
Furthermore assume that \(\widehat{\varphi}(\alpha)=bc=(b_1c_1,b_2c_2)=(b_1c_1,1)\cdot (1,b_2c_2)\).
Since \(\varphi_i\) is a transfer homomorphism for \(i=\{1,2\}\), there exist \(a_i, d_i\in T_i\), and \(\epsilon_i\in T_i^\times\), such that \(\varphi_i(a_i)=b_i\epsilon_i\), and \(\alpha_i=a_id_i\).
Hence \(\alpha=(a_1,a_2)(d_1,d_2)\) and \(\widehat{\varphi}(a_1,a_2)=b\cdot(\epsilon_1,\epsilon_2)\).
We conclude that \(\widehat{\varphi}\) is a transfer homomorphism.
\end{proof}

\subsection{Orders in Central Simple Algebras}
Central simple algebras have been studied in great detail over the past century.
We recall some common results and definitions within this field, and we refer to \cite{reinermaximal} for a detailed background of this theory.

Let \(R\) be a Dedekind domain with field of fractions \(K\).
Let \(A/K\) be a finite dimensional central simple algebra.
The dimension \(\dim_K(A)=n^2\) of \(A\) as a \(K\)-vector space is always a square, and the integer \(n\) is known as the \textit{degree} of \(A/K\).
By the Wedderburn-Artin theorem \(A\) is isomorphic to a matrix ring over a  division ring with center \(K\).
Write \(A=M_m(D)\) where \(D\) is the central division ring of dimension, and write \(d^2\) for the dimension of \(D\).
By \cite[Theorem 7.15]{reinermaximal}, the division ring \(D\) contains a degree \(d\) field extension \(L/K\), such that there exists an isomorphism of \(L\)-algebras \(A\tensor_K L\cong M_n(L)\).
The field \(L\) is a \textit{splitting field} of \(A\) and comes equipped with an embedding \(\mu:A\to M_n(L)\).
For any \(\alpha\in A\) we define the \textit{reduced characteristic polynomial of} \(\alpha\) as the characteristic polynomial of \(\mu(\alpha)\).
We write \(f_\alpha(X)\) for this polynomial, and note that it is a monic polynomial of degree \(n\).
By \cite[Theorem 9.3]{reinermaximal} the polynomial \(f_\alpha(X)\) lies inside \(K[X]\) for all \(\alpha\in A\), and is independent of the field \(L\) and the embedding \(\mu\).
As the embedding \(\mu\) is injective it further follows from the Cayley-Hamilton theorem that \(f_\alpha(\alpha)=0\) for all \(\alpha\in A\).
We define the \textit{reduced norm map} of \(A\) as 
\[\nrd:A\to K, \alpha\mapsto \det(\mu(\alpha)).\]
The reduced norm map is multiplicative and does not depend on \(L\) or \(\mu\).
Writing \(f_\alpha(X)=\sum_{i=0}^n a_i X^i\) with \(a_n=1\), we find that \(\nrd (\alpha)=(-1)^na_0\).
An \(R\)-\textit{order} of \(A\) is a subring \(\roi\subset A\) that is finitely generated as an \(R\)-module and satisfies \(K\roi=A\).

\begin{lemma}\label{lem:nrdatomicity}
Let \(\roi\subset A\) be an \(R\)-order, and let \(\alpha\in \roi\). The following statements hold:
\begin{enumerate}
\item \(\nrd \alpha\in R\),
\item \(\alpha\in \roi^\bullet\) if and only if \(\nrd \alpha\neq 0\),
\item \(\alpha\in \roi^\times\) if and only if \(\nrd \alpha\in R^\times\),
\item \(\roi^\bullet\) is atomic.
\end{enumerate}
\end{lemma}

\begin{proof}
Write \(f_\alpha(X)=\sum_{i=0}^n a_iX^i\) for the reduced characteristic polynomial of \(\alpha\).
The reduced norm of \(\alpha\) is given by 
\begin{align}\label{eq:nrdcharpoly}
\nrd (\alpha)=(-1)^{n-1}\sum_{i=1}^n a_i\alpha^{i}.
\end{align}
By \cite[Theorem 10.1]{reinermaximal} the polynomial \(f_\alpha\) has coeffcients in \(R[X]\), and 1 follows.
Equation \eqref{eq:nrdcharpoly} shows that \(\alpha\) is a right divisor of \(\nrd \alpha\).
Let \(b,c\in \roi\).
If \(\alpha b=\alpha c\) then  \(\nrd (\alpha)b =\nrd (\alpha) c\).
If \(\nrd \alpha\neq 0\), then it is cancellative hence so is \(\alpha\).
If \(\nrd \alpha =0\) then \(\alpha\) is a zero divisor and therefore cannot be cancellative.
This proves 2.
Next assume \(\alpha\in \roi^\bullet\).
By equation \eqref{eq:nrdcharpoly} the inverse of \(\alpha\) is given by \(\nrd(\alpha)^{-1}(-1)^{n-1}\sum_{i=1}^na_i\alpha^{i-1}\).
Hence \(\alpha\) is invertible in \(\roi\) if \(\nrd \alpha\) is invertible in \(R\).
Conversely assume there exists a \(\beta\in \roi^\bullet\) such that \(\alpha\beta=1\).
Then \(\nrd(\alpha)\nrd(\beta)=\nrd(\alpha\beta)=1\) hence \(\nrd\alpha\) is invertible in \(R^\times\). This proves 3.
To prove 4, we observe that \(\roi\) is noetherian, hence the claim follows by \cite[Proposition 3.1]{smertnig2013}.
\end{proof}

For the remainder of this paper we assume that \(K\) is a number field.
Let \(A\) be a central simple \(K\)-algebra of degree \(n\).
For a place \(\mathfrak{p}\) of \(K\), we define the \textit{completion} of \(A\) at \(\mathfrak{p}\) as \(A_\mathfrak{p}=A\tensor_K K_\mathfrak{p}\).
Then \(A_\mathfrak{p}\) is a central simple \(K_\mathfrak{p}\)-algebra of degree \(n\).
We say that \(A\) is \textit{split} at \(\mathfrak{p}\), if \(A_\mathfrak{p}\cong M_n(K_\mathfrak{p})\), else \(A\) is said to be \textit{ramified} at \(\mathfrak{p}\).
The algebra \(A\) is ramified at only finitely many places of \(K\) \cite[Theorem 25.7]{reinermaximal}, and we write \(\text{Ram}(A)\) for the set of places where \(A\) is ramified.
The algebra \(A\) is said to be \textit{totally ramified} at \(\mathfrak{p}\) if \(A_\mathfrak{p}\) is a division ring, and we write \(\text{TRam}(A)\) for the set of places where \(A\) is totally ramified.
Let \(\roi\subset A\) be an \(R\)-order, and let \(\mathfrak{p}\) be a non-zero prime of \(R\).
We write \(R_\mathfrak{p}\) for the completion of \(R\) at \(\mathfrak{p}\), and \(R_{(\mathfrak{p})}\subset K\) for its localization.
We define the completion of \(\roi\) at \(\mathfrak{p}\) as \(\roi_\mathfrak{p}=\roi\tensor_R R_\mathfrak{p}\), and note that it is an \(R_\mathfrak{p}\)-order in \(A_\fp\).
Furthermore, \(\roi_\fp\) is a maximal \(R_\fp\)-order for all but finitely many primes \(\mathfrak{p}\) \cite[Lemma 9.5.3]{Voightquatalg}.

Let \(\mathfrak{p}\) be a non-zero prime of \(R\) and let \(D_\mathfrak{p}\) be a division ring over \(K_\mathfrak{p}\). 
Write \(d\) for the degree of \(D_\mathfrak{p}\) over \(K_\mathfrak{p}\).
Fix a uniformizer \(\pi\) for the maximal ideal of \(R_\mathfrak{p}\), and let \(v_\mathfrak{p}:K_\mathfrak{p}\to \Z\cup \{\infty\}\) denote the valuation map.
As \(D_\mathfrak{p}\) is a central simple algebra over \(K_\mathfrak{p}\), the reduced norm map allows us to define a valuation on \(D_\mathfrak{p}\) given by \(w_\mathfrak{p}=v_\mathfrak{p}\circ \nrd\).
Note that \(w_\mathfrak{p}(\pi)=d\).
The ring \(D_\mathfrak{p}\) contains a unique maximal order \(\Delta_\mathfrak{p}\) given by 
\[\Delta_\mathfrak{p}=\{\alpha\in D_\mathfrak{p}\mid w_\mathfrak{p}(\alpha)\geq 0\}.\]
It has a unique maximal two-sided ideal \(\mathfrak{P}=\{\alpha\in D_\mathfrak{p}\mid w_\mathfrak{p}(\alpha)>0\}\) which is also the unique maximal left-ideal as well as the unique maximal right-ideal of \(\Delta_\mathfrak{p}\).
It is a principal ideal and after choosing a generator \(\api\), any \(\alpha\in \Delta_\mathfrak{p}\) may be written uniquely as \(\alpha=u\api^{w_\mathfrak{p}(\alpha)}\) with \(u\in \Delta_\mathfrak{p}^\times\).

\begin{lemma}\label{cor:invertnorm}
Let \(B,C\in M_m(\Delta_\mathfrak{p})^\bullet\) and assume that
\[B=C+ \api^tM\]
for some \(t\in \N\) and \(M\in M_m(\Delta_\mathfrak{p})\).
Suppose that either \(t>w_\mathfrak{p}(\nrd B)\), or \(B\in M_r(R[\api])\) and \(t>v_\mathfrak{p}(\nrd B)\).
Then \(C=\gamma B\) for some \(\gamma\in M_m(\Delta_\mathfrak{p})^\times\).
\end{lemma}

\begin{proof}
Write \(f=\sum_{i=0}^{md} a_iX^i\) for the reduced characteristic polynomial of \(B\).
As \(B\) is cancellative, it is invertible in \(M_m(D_\mathfrak{p})\), and its inverse is given by
\[B^{-1}= \frac{(-1)^{md-1}}{\nrd B}\sum_{i=1}^{md}a_i B^{i-1}.\]
It follows that \(B^{-1}\in \api^{-w_\mathfrak{p}(\nrd B)}M_m(\Delta_\mathfrak{p})\).
Let us first assume that \(t>w_\mathfrak{p}(\nrd B)\).
Then \(\api^t MB^{-1}\) is contained in \(\api M_m(\Delta_\mathfrak{p})\), hence \(CB^{-1}=I_r-\api^t MB^{-1}\) is contained in \(M_m(\Delta_\mathfrak{p})\).
As \(\api\in \text{rad}(M_m(\Delta_\mathfrak{p}))\) and \(CB^{-1}\) is invertible in \(M_m(\Delta_\mathfrak{p})/\text{rad}(M_m(\Delta_\mathfrak{p}))\) it follows that \(CB^{-1}\) is invertible in \(M_m(\Delta_\mathfrak{p})\)  and we have \(C=CB^{-1}B\).

If \(B\in M_m(R[\api])\) then the coefficients of \(B\) commute.
In particular \(B\) has a well defined determinant in \(\Delta_\mathfrak{p}\), as well as an adjugate inside \(M_m(\Delta_\mathfrak{p})\).
Its inverse is then given by \(B^{-1}=\det(B)^{-1}\text{Adj}(B)\in \api^{-w_\mathfrak{p}(\det B)}M_m(\Delta_\mathfrak{p})\).
As \(w_\mathfrak{p}(\det B)=v_\mathfrak{p}(\nrd B)\), the claim follows analogously.
\end{proof}

Let \(\tilde{m}=(m_1,...,m_r)\) be a partition of \(m\).
We can view the elements of \(M_m(D_\mathfrak{p})\) as matrices in block form, and observe the vector space \(M_m(D_\mathfrak{p})\) naturally decomposes via
\begin{align}\label{eq:vectordecomp}
\Phi:M_m(D_\mathfrak{p})\xrightarrow{\sim} \bigoplus_{k,l=1}^r M_{m_k\times m_l}(D_\mathfrak{p}).
\end{align}
For any \(B=[B_{m_k\times m_l}]_{k,l}\) and \(C=[C_{m_k\times m_l}]_{k,l}\) in this direct sum we define their product
\[\left[ [B_{m_k\times m_l}]\cdot [C_{m_k\times m_l}]\right]_{p,q}=\sum_{i=1}^rB_{m_p\times m_i}C_{m_i\times m_q}.\]
turning \(\Phi\) into an algebra isomorphism in a natural way.
A matrix \(B=[B_{m_k\times m_l}]_{k,l}\) is said to be \textit{block diagonal} if \(B_{m_k\times m_l}=0\) for all pairs \(k\neq l\).

\begin{lemma}\label{lem:block diagonal}
Let \(m\in \Z_{\geq 1}\) be given and let \((m_1,...,m_r)\) be a partition of \(m\). 
Let \(B,C\in M_m(\Delta_\mathfrak{p})^\bullet\) be given and assume that \(B\) is block diagonal.
Then \(B\) is a right divisor of \(C\) if and only if \(B_{m_l\times m_l}\) is a right divisor of \(C_{m_k\times m_l}\) for all pairs \(k,l\). 
\end{lemma}
\begin{proof}
This follows directly from matrix multiplication.
\end{proof}

\begin{definition}\label{def:hermite}
Let \(\roi\subset A\) be an \(R\)-order.
Then
\begin{enumerate}
\item \(\roi\) is \textit{Hermite} if every stably free, locally free right \(\roi\)-module is free.
\item \(\roi\) is \textit{hereditary} if every right ideal of \(\roi\) is projective.
\end{enumerate}
\end{definition}

Combining Jacobinski's Cancellation Theorem (\cite[Theorem 51.24]{curtisreiner}), with \cite[Proposition 51.2]{curtisreiner} shows that every order \(\roi\) inside a central simple algebra \(A/K\) of degree \(n>2\) is Hermite.
Being hereditary is a local property, that is \(\roi\) is hereditary if and only if the same holds for the completions \(\roi_\mathfrak{p}\) for all non-zero primes \(\mathfrak{p}\) of \(R\).
Hereditary orders have been studied in great detail and we refer to \cite[Part 9]{reinermaximal} for more background on this topic.
Every maximal order is hereditary, and since a global order \(\roi\subset A\) is maximal at all but finitely many non-zero primes of \(R\) the set of primes where \(\roi_
\mathfrak{p}\) is not hereditary is finite.
We conclude this subsection with an important consequence of the local norm map with respect to factorization theory due to  D. Estes

\begin{lemma}[{\cite{estes1991}}]\label{lem:hereditarytransfer}
Let \(\roi_\mathfrak{p}\subset M_m(D_\mathfrak{p})\) be a hereditary \(R_\mathfrak{p}\)-order.
The reduced norm map \(\nrd: \roi_\mathfrak{p}^\bullet\to R_\mathfrak{p}^\bullet\) is a surjective transfer homomorphism.
In particular \(\rho(\roi_\mathfrak{p}^\bullet)=1\) for all hereditary orders.
\end{lemma}

\subsection{Adeles}
Let \(K\) be a number field with ring of integers \(R\).
Let \(S\subset \text{Pl}(K)\) be a finite set of places of \(K\).
The set \(S\) is said to be \textit{eligible} if it contains all archimedean places of \(K\).
Let \(R_S\) denote the ring of \(S\)-integers in \(K\),
that is \(R_S=\bigcap_{\mathfrak{p}\notin S}R_{(\mathfrak{p})}\), and let \(A/K\) be a central simple algebra of degree \(n\).

Given an \(R_S\)-order \(\roi\subset A\), we define the \textit{monoid of \(S\)-adeles of \(\roi\)} as
\[\widehat{\roi}=\sideset{}{'}\prod_{\mathfrak{p}\notin S} \roi_{\mathfrak{p}}= \left\{(\alpha_{\mathfrak{p}})_\mathfrak{p}\in \prod_{\mathfrak{p}\notin S}\roi_{\mathfrak{p}}~\middle| ~\alpha_{\mathfrak{p}}\in \roi_{\mathfrak{p}}^\times\text{ for all but finitely many }\mathfrak{p}\right\},\]
and we define \(\widehat{R}\) analogously.
Its submonoid \(\widehat{\roi}^\bullet\) of cancellative elements is given by the restricted product \(\prod_{\mathfrak{p}\notin S}' \roi_{\mathfrak{p}}^\bullet\), and its invertible elements are given by \(\widehat{\roi}^\times= \prod_{\mathfrak{p}\notin S} \roi_\mathfrak{p}^\times\).
Since \(\roi_\mathfrak{p}^\bullet\) is atomic for all \(\mathfrak{p}\notin S\), so is \(\widehat{\roi}^\bullet\) by Lemma \ref{lem:restricttransfer}.
We further define the \textit{\(S\)-adele ring of \(A\)} as
\[\widehat{A}=\sideset{}{'}\prod_{\mathfrak{p}\notin S} A_{\mathfrak{p}}= \left\{(\alpha_{\mathfrak{p}})_\mathfrak{p}\in \prod_{\mathfrak{p}\notin S}A_{\mathfrak{p}}~\middle| ~\alpha_{\mathfrak{p}}\in \roi_{\mathfrak{p}}\text{ for all but finitely many }\mathfrak{p}\right\},\]
and define \(\widehat{K}\) anologously.
This definition is independent of the chosen order \(\roi\), as the completions of any two orders \(\roi,~\roi'\subset A\) differ at only finitely many non-zero primes of \(R_S\) \cite[Theorem 9.4.9]{Voightquatalg}.
For every prime \(\mathfrak{p}\notin S\) we can choose a uniformizer \(\pi_\mathfrak{p}\in R\) for the maximal ideal of \(R_\mathfrak{p}\).
Since \(K_\mathfrak{p}^\times= R_\mathfrak{p}^\times\pi_\mathfrak{p}^\Z=K^\times R_\mathfrak{p}^\times\), and since \(R_\mathfrak{p}^\times \subset \roi_\mathfrak{p}^\times\), we conclude that \(K\roi_\mathfrak{p}=A_\mathfrak{p}\).
In particular we have \(K^\times \widehat{\roi}^\bullet=\widehat{A}^\times\).

The reduced norm maps \(\nrd:\roi_\mathfrak{p}^\bullet\to R_\mathfrak{p}^\bullet\) give rise to a reduced norm map \(\nrd:\widehat{\roi}^\bullet\to \widehat{R}^\bullet\).
The cancellative monoid \(\roi^\bullet\) naturally embeds into \(\widehat{\roi}^\bullet\), and the reduced norm map \(\nrd:\roi^\bullet\to R_S^\bullet\) and the diagram
\[\begin{tikzcd} \roi^\bullet \arrow[r, "\nrd"] \arrow[d, hook] & R_S^\bullet \arrow[d, hook]\\
\widehat{\roi}^\bullet \arrow[r,"\nrd"] & \widehat{R}^\bullet
\end{tikzcd}\]
commutes.
We further claim that the embedding \(\roi^\bullet\hookrightarrow \widehat{\roi}^\bullet\) is a left divisor homomorphism.
That is, for any \(a,b\in \roi^\bullet\) such that \(b^{-1}a\in\widehat{\roi}^\bullet\), we have \(b^{-1}a\in \roi^\bullet\).

\begin{lemma}\label{lem:leftdiv}
Let \(\roi\subset A\) be an \(R_S\)-order.
The natural map \(\roi^\bullet\to \widehat{\roi}^\bullet, \alpha \mapsto (\alpha_{\mathfrak{p}})_{\mathfrak{p}}\) is a left divisor homomorphism.
\end{lemma}

\begin{proof}
Let \(\mathfrak{p}\) be a non-zero prime of \(R_S\).
As \(\roi_{(\mathfrak{p})}\) is an \(R_{(\mathfrak{p})}\)-lattice,
it follows from \cite[Lemma 9.5.3]{Voightquatalg} that \(\roi_{\mathfrak{p}}^\bullet \cap A^\times = \roi_{(\mathfrak{p})}^\bullet\).
Let \(a,b\in \roi^\bullet\) and assume that \(b^{-1}a\in\roi_{\mathfrak{p}}^\bullet\) for all \(\mathfrak{p}\).
Since \(b^{-1}a\in A\), it follows that \(b^{-1}a\in \roi_{(\mathfrak{p})}^\bullet\) for all primes \(\mathfrak{p}\), and thus that
\[b^{-1}a\in \bigcap_{\mathfrak{p}\notin S}\roi_{(\mathfrak{p})}^\bullet =\roi^\bullet.\]
\end{proof}

The quotient \(\widehat{K}^\times/\widehat{R}^\times\) is naturally isomorphic to the group of fractional \(R_S\)-ideals, under which the principal \(R_S\)-ideals coincide with the image of \(K^\times\) in \(\widehat{K}\).
This allows for a natural group isomorphism 
\[\text{Cl}(R_S)\xrightarrow{\sim}K^\times\backslash \widehat{K}^\times/\widehat{R}^\times.\]
Since \(K^\times \widehat{R}^\bullet=\widehat{K}^\times\), we may replace \(\widehat{K}^\times\) with \(\widehat{R}^\bullet\) in this quotient.
Similarly the quotient \(\widehat{A}^\times/ \widehat{\roi}^\times\) may be identified with the set of locally principal right fractional \(\roi\)-ideals \cite[Lemma 27.6.8] {Voightquatalg}.
Any such two ideals \(I,J\) are in the same \textit{right class} if there exists an \(\alpha\in A^\times\) such that \(\alpha I=J\).
This defines an equivalence relation and the set of equivalence classes is known as the \textit{right class set} of \(\roi\).
We write \(\text{Cls}(\roi)\) for this set, and note that we have a natural identification
\[\text{Cls}(\roi)\xrightarrow{\sim} A^\times\backslash \widehat{A}^\times /\widehat{\roi}^\times.\]
As \(K^\times \widehat{\roi}^\bullet=\widehat{A}^\times\) we may replace \(\widehat{A}^\times\) with \(\widehat{\roi}^\bullet\) in this quotient.
The reduced norm map allows us to relate the right class set of \(\roi\) with a ray class group of \(K\).
To this end we define
\[K_A^\times = \{a\in K^\times \mid a_v>0 \text{ for all archimedean } v\in \text{Ram}(A)\}.\]

\begin{lemma}\label{lem:nrdsurjclass}
Let \(\roi\subset A\) be an \(R_S\)-order.
The reduced norm map induces a  well-defined surjective map of pointed sets
\begin{align}\label{eq:nrdsurjclass}
\nrd: A^\times \backslash\widehat{\roi}^\bullet/\widehat{\roi}^\times\to K_A^\times\backslash \widehat{R}^\bullet/\nrd(\widehat{\roi}^\times).
\end{align}
The latter is a finite abelian group that may be realized as a ray class group over \(K\).
In particular every class contains infinitely many primes.
\end{lemma}

\begin{proof}
We may replace the factors \(\widehat{\roi}^\bullet\) and \(\widehat{R}^\bullet\) in equation \eqref{eq:nrdsurjclass} with \(\widehat{A}^\times\) and \(\widehat{K}^\times\) respectively.
By the Hasse-Schilling-Maass norm theorem \cite[Theorem 33.15]{reinermaximal},we have \(\nrd A^\times=K_A^\times\).
Surjectivity of the map of pointed sets now follows as the reduced norm map \(\nrd: \widehat{A}^\times \to \widehat{K}^\times\) is surjective.

Combining Theorem 27.5.10 and Remarks 27.5.12, 27.5.9 of \cite{Voightquatalg},
shows that \(K_A^\times \backslash \widehat{R}^\bullet/\nrd(\widehat{\roi}^\bullet)\) can be realized as a ray class group of \(K\), if \(\nrd (\roi^\times)\) is a finite index open subgroup of \(\widehat{R}^\times\).
If \(\roi_\mathfrak{p}\) is a maximal order, then \(\nrd (\roi_\mathfrak{p}^\times) =R_\mathfrak{p}^\times\).
As \(\roi_\mathfrak{p}\) is maximal at all but finitely many non-zero primes of \(R_S\), it remains to show that \(\nrd(\roi_\mathfrak{p}^\times)\) is a finite index open subgroup of \(R_\mathfrak{p}^\times\).
Let \(\roi'_\mathfrak{p}\subset A_\mathfrak{p}\) be a maximal order containing  \(\roi_\mathfrak{p}\).
Let \(\pi\) be a uniformizer for the maximal ideal of \(R_\mathfrak{p}\).
There exists an \(n\in \N\) such that 
\[1+\pi^{n+1}\roi'_\mathfrak{p}\subset 1+\pi\roi_\mathfrak{p}\subset \roi_\mathfrak{p}^\times \subset \roi_\mathfrak{p}'^\times.\]
On the other hand we have inclusions
\[1+\pi^{n+1}\roi'_\mathfrak{p}\subset 1+\pi\roi'_\mathfrak{p}\subset \roi_\mathfrak{p}'^\times.\]
Note that we have \([1+\pi\roi'_\mathfrak{p}:1+\pi^{n+1}\roi'_\mathfrak{p}]=[\roi'_\mathfrak{p}:\pi^n\roi'_\mathfrak{p}]<\infty\).
Furthermore consider the reduction map \(\roi'_\mathfrak{p}\to\roi'_\mathfrak{p}/\pi\roi'_\mathfrak{p}\).
Restricting to units, we obtain a short exact sequence 
\[1\to 1+\pi\roi'_\fp\to \roi_\fp'^\times\to \left( \roi'_\mathfrak{p}/\pi\roi'_\mathfrak{p}\right)^\times\to 1.\]
Hence the index \([\roi_\fp'^\times:1+\pi\roi'_\fp]\) is finite as well.
It follows \(\roi_\mathfrak{p}^\times\) is a finite index open subgroup of \(\roi_\mathfrak{p}'^\times\) proving the claim.
The final claim is a direct consequence of the Chebotarëv Density Theorem.
\end{proof}

\begin{definition}
Let \(\roi\subset A\) be an \(R_S\)-order.
The  \textit{ray class group of \(K\) associated to \(\roi\)} is the class group
\(K_A^\times\backslash \widehat{R}^\bullet/\nrd(\widehat{\roi}^\times).\)
\end{definition}

\section{The Adelic Approach}
Throughout this section we fix a number field \(K\) with ring of integers \(R\), and let \(A/K\) be a central simple algebra of degree \(n\).
Further we fix an eligible set of places \(S\subset \text{Pl}(K)\), and we fix an \(R_S\)-order \(\roi\subset A\).
Throughout this paper, \(\widehat{\roi}\) and \(\widehat{A}\) will refer to monoid of \(S\)-adeles of \(\roi\) and the \(S\)-adele ring of \(A\) respectively.

Let \(\mathfrak{M}\) denote the finite set of non-zero primes of \(R_S\) where \(\roi_\mathfrak{p}\) is not hereditary.
Define
\begin{align*}
\widetilde{\roi}=\sideset{}{'}\prod_{\mathfrak{p}\notin \mathfrak{M}\cup S} \roi_{\mathfrak{p}}, && \roi_\mathfrak{M}=\prod_{\mathfrak{p}\in\mathfrak{M}} \roi_{\mathfrak{p}}.
\end{align*}
This allows us to identify \(\widehat{\roi}=\widetilde{\roi}\times \roi_\mathfrak{M}\).
The goal of this section is to prove the following theorem:

\begin{theorem}\label{prop:transfer}
Let \(\roi\subset A\) be a Hermite \(R_S\)-order and let \(\mathfrak{M}\subset\text{Pl}(K)\backslash S\) be the set of non-zero primes of \(R_S\) where \(\roi_\mathfrak{p}\) is not hereditary.
Let \(\mathcal{C}=K_A^\times\backslash\widehat{R}^\bullet/\nrd(\widehat{\roi}^\times)\) be the ray class group of \(K\) associated to \(\roi\).
The reduced norm map \(\iota: \roi_\mathfrak{M}^\bullet\to \mathcal{C}\) induces a transfer homomorphism 
\[ \varphi:\roi^\bullet \to \mathcal{B}_{\roi_\mathfrak{M}}(\mathcal{C},\iota).\]
\end{theorem}

The monoid \(\mathcal{B}_{\roi_\mathfrak{M}}(\mathcal{C},\iota)\) described in Theorem \ref{prop:transfer} will be referred to as the \textit{\(T\)-block monoid associated to \(\roi\).} 
If  \(\nrd(\widehat{\roi})^\times=\widehat{R}^\times\), we have \(\mathcal{C}=\text{Cl}_A(R_S)\).
If \(\roi\) is  moreover assumed to be hereditary, then \(\mathcal{B}_{\roi_\mathfrak{M}}(\mathcal{C},\iota)=\mathcal{B}(\mathcal{C})\).
A result also seen in \cite[Corollary 5.24]{smertnig16}.

For every prime \(\mathfrak{p}\notin \mathfrak{M}\cup S\) we fix a uniformizer \(\pi_\mathfrak{p}\) of \(R_\mathfrak{p}\) and let \(\mathcal{P}=\{\pi_\mathfrak{p}\mid \mathfrak{p}\notin \mathfrak{M}\cup S\}\).

\begin{lemma}\label{lem:nrddiag}
The following statements hold:
\begin{enumerate}
\item The reduced norm map gives rise to a commutative diagram of pointed sets:
\begin{equation*}
\begin{tikzcd}
\widehat{\roi}^\bullet \arrow[r,"\text{nrd}"]\arrow[d]&\widetilde{R}^\bullet \times \nrd(\roi_\mathfrak{M}^\bullet)\arrow[d]\\
\widehat{\roi}^\bullet/\widehat{\roi}^\times \arrow[r, "\text{nrd}"] & \widetilde{R}^\bullet/\widetilde{R}^\times\times \nrd(\roi_\mathfrak{M}^\bullet/\roi_\mathfrak{M}^\times).
\end{tikzcd}\label{diag:cominf}
\end{equation*}
Furthermore \(\widetilde{R}^\bullet/\widetilde{R}^\times\) is isomorphic to the free abelian monoid \(\mathcal{F}(\mathcal{P})\) and the composition \(\widetilde{\roi}^\bullet\to \mathcal{F}(\mathcal{P})\) is a surjective transfer homomorphism.

\item The reduced norm map gives rise to a commutative diagram of pointed sets:
\begin{equation*}
\begin{tikzcd}
\widehat{\roi}^\bullet \arrow[r,"\text{nrd}"]\arrow[d]&\widehat{R}^\bullet \arrow[d]\\
A^\times \backslash\widehat{\roi}^\bullet/\widehat{\roi}^\times \arrow[r, "\text{nrd}"] & K_A^\times\backslash \widehat{R}^\bullet/\nrd(\widehat{\roi}^\times),
\end{tikzcd}
\end{equation*}
sending \(\roi^\bullet\) to the trivial class of \(K_A^\times\backslash \widehat{R}^\bullet/\nrd(\widehat{\roi}^\times)\).
Furthermore \(K_A^\times\backslash \widehat{R}^\bullet/\nrd(\widehat{\roi}^\times)\) is a finite abelian group that can be realized as a ray class group over \(K\). 
In particular every class contains infinitely many primes.
\end{enumerate}

\end{lemma}

\begin{proof}
The commutativity of the first diagram is direct as the norm map is multiplicative and sends units to units.
The surjectivity of the reduction maps is direct as well.
Since \(\roi_\mathfrak{p}\) is hereditary for all \(\mathfrak{p}\notin\mathfrak{M}\cup S\), the reduced norm is a surjective transfer homomorphism by Lemma \ref{lem:hereditarytransfer}, hence so is the reduced norm on \(\widetilde{\roi}^\bullet\) by Lemma \ref{lem:restricttransfer}.
To see that \(\widetilde{R}^\bullet/\widetilde{R}^\times\) carries the structure of a free abelian monoid over \(\mathcal{P}\), it suffices to remark that any \(\alpha\in \widetilde{R}^\bullet\) can be written uniquely as \(\alpha=(\mu_\mathfrak{p}\pi_\mathfrak{p}^{k_\mathfrak{p}})_\mathfrak{p}\), with \(\mu_\mathfrak{p}\in R_\mathfrak{p}^\times\), and \(k_\mathfrak{p}\in \Z_{\geq 0}\) such that \(k_\mathfrak{p}=0\) for all but finitely many \(\mathfrak{p}\).
The reduction map \(\widetilde{R}^\bullet\to \widetilde{R}^\bullet/\widetilde{R}^\times\) is a surjective transfer homomorphism, and since the composition of two surjective transfer homomorphisms is again a surjective transfer homomorphism, this proves the first statement.

The commutativity of the second diagram follows from the commutativity of the first diagram and the Hasse-Schilling-Maass norm theorem \cite[Theorem 33.15]{reinermaximal}.
The surjectivity of the lower norm map, as well as the fact that \(K_A^\times\backslash \widehat{R}^\bullet/\nrd(\widehat{\roi}^\times)\) is a finite abelian group that can be realized as a ray class group over \(K\), follow from Lemma \ref{lem:nrdsurjclass}.
Finally, to see that \(\roi^\bullet\) maps to the trivial class, we note that we may replace the instances of \(\widehat{\roi}^\bullet\) and \(\widehat{R}^\bullet\) with \(\widehat{A}^\times\) and \(\widehat{K}^\times\) respectively.
The image of \(\roi^\bullet\) lies inside \(A^\times\), hence it lies in the trivial class of \(A^\times\backslash \widehat{\roi}^\bullet /\widehat{\roi}^\times\).
As the reduced norm map is a pointed map, we conclude that \(\nrd(\roi^\bullet)\) lies in the trivial class of  \(K_A^\times\backslash \widehat{R}^\bullet/\nrd(\widehat{\roi}^\times)\).
\end{proof}

\begin{remark}\label{rem:transfertoclassgroup}
 Write \(\mathcal{C}\) for the ray class group \(K_A^\times\backslash \widehat{R}^\bullet/\nrd(\widehat{\roi}^\times)\).
Since every class of \(\mathcal{C}\) contains infinitely many primes, the reduction map \(\widetilde{R}^\bullet/\widetilde{R}^\times \to \mathcal{C}\) is surjective.
Since \(\widetilde{R}^\bullet/\widetilde{R}^\times\) is a free abelian monoid over \(\mathcal{P}\), the natural map \(\mathcal{F}(\mathcal{P})\to \mathcal{F}(\mathcal{C})\) is a surjective transfer homomorphism.
Hence the norm map induces a surjective transfer homomorphism \(\widetilde{\roi}^\bullet\to \mathcal{F}(\mathcal{C})\).
\end{remark}
 
Given an \(\alpha\in \widehat{\roi}^\bullet\), we write \([\alpha]_\roi\) for its class in \(\text{Cls}(\roi)=A^\times \backslash\widehat{\roi}^\bullet/\widehat{\roi}^\times\).
Similary for any \(a\in \widehat{R}^\bullet\) we write \([a]_R\) for its class inside \(\mathcal{C}\).
 
\begin{lemma}\label{lem:hermite}
Assume that \(\roi\) is Hermite.
Then for every \(\alpha\in \widehat{\roi}^\bullet\) such that \([\nrd\alpha]_R=[1]_R\), there exists an \(\epsilon\in\widehat{\roi}^\times\) such that \(\alpha\epsilon\in \roi^\bullet\).
\end{lemma}
\begin{proof}
Assume first that \(n=2\).
By \cite[Proposition 4.4]{smertnigvoight2019}, and \cite[Proposition 6.3]{smertnig2013}, \(\roi\) is Hermite if and only if \(\nrd^{-1}([1]_R)=[1]_\roi=A^\times\widehat{\roi}^\times\). 
Hence any \(\alpha\in \widehat{\roi}^\bullet\) satisfies \([\nrd \alpha]_R=[1]_R\), if and only if \(\alpha=a\epsilon\) for some \(a\in A^\times\) and \(\epsilon \in \widehat{\roi}^\times\).
As \(a=\alpha\epsilon^{-1}\in \widehat{\roi}^\bullet \cap A^\times=\roi^\bullet\) this proves the first statement.
If \(n>2\) then the norm map described in Lemma \ref{lem:nrdsurjclass} is a bijection by \cite[Theorem 2]{frohlich1975} and the result follows analogously.
\end{proof}

Given Lemma \ref{lem:hermite} we are now ready prove Theorem \ref{prop:transfer}.

\begin{proof}[Proof of Theorem \ref{prop:transfer}]
Let \(\mathfrak{M}\) denote the finite set of non-zero primes of \(R_S\) where \(\roi\) is not hereditary.
Let \(\mathcal{P}=\{\pi_\mathfrak{p}\mid \mathfrak{p}\notin \mathfrak{M}\cup S\}\) be a set of uniformizers for the maximal ideals of \(R_\mathfrak{p}\) for all \(\mathfrak{p}\notin \mathfrak{M}\cup S\).
Let \(\mathcal{C}\) be the ray class group of \(K\) associated to \(\roi\), and identify \(\widehat{\roi}=\widetilde{\roi}\times \roi_\mathfrak{M}\).
By Remark \ref{rem:transfertoclassgroup} there exists a transfer homomorphism \(\widetilde{\roi}^\bullet\to \mathcal{F}(\mathcal{C})\).
Extending this with the identity on \(\roi_\mathfrak{M}^\bullet\) we find a transfer homomorphism \[\phi:\widehat{\roi}^\bullet\to \mathcal{F}(\mathcal{C})\times \roi_\mathfrak{M}^\bullet.\]
Notice that any \(\epsilon=(\tilde{\epsilon},\epsilon_\mathfrak{M})\in \widehat{\roi}^\times=\widetilde{\roi}^\times\times \roi_\mathfrak{M}^\times\) maps to \((1,\epsilon_{\mathfrak{M}})\) under this map.
Let \(\iota: \roi_\mathfrak{M}^\bullet \to \mathcal{C}\) be the norm map and define \(\mathcal{B}=\mathcal{B}_{\roi_\mathfrak{M}}(\mathcal{C},\iota)\).
Its unit group is given by \(\mathcal{B}^\times=\{1\}\times \roi_\mathfrak{M}^\times\), and by Lemma \ref{lem:nrddiag} the restriction of \(\phi\) to \(\roi^\bullet\) maps into \(\mathcal{B}\).
We denote \(\varphi\) for this restriction and we claim that it is a transfer homomorphism.
Let \(\alpha\in \widehat{\roi}^\bullet\).
By construction we have \(\phi(\alpha)\in \mathcal{B}\) if and only if \([\nrd \alpha]_R=[1]_R\),
hence by Lemma \ref{lem:hermite} we must have \(\alpha=a\epsilon\) for some \(a\in \roi^\bullet\) and \(\epsilon\in \widehat{\roi}^\times\).
Writing \(\epsilon=(\tilde{\epsilon},\epsilon_\mathfrak{M})\) it follows that \(\phi(\alpha)=\varphi(a)\cdot(1,\epsilon_\mathfrak{M})\) and hence that \(\mathcal{B}=\varphi(\roi^\bullet)\mathcal{B}^\times\).
Since \(\phi\) is a surjective transfer homomorphism we also have
\[\varphi^{-1}(\mathcal{B}^\times)\subset \roi^\bullet \cap \phi^{-1}(\{1\}\times \roi_\mathfrak{M}^\times)=\roi^\bullet \cap \widehat{\roi}^\times =\roi^\times.\]
In order to complete the proof let \(a\in \roi^\bullet\) be given and assume \(\varphi(a)\) factors as \(\varphi(a)=(T_1,\alpha_{\mathfrak{M}})\cdot (T_2,\beta_\mathfrak{M})\).
We wish to show that \(a\) has a left divisor \(a_1\in \roi^\bullet\) such that \(\varphi(a_1)=(T_1,\alpha_\mathfrak{M}\epsilon)\) for some \(\epsilon\in \roi_\mathfrak{M}^\times\).
Write \(a=(\tilde{a},\alpha_\mathfrak{M}\beta_\mathfrak{M})\), with \(\tilde{a}=(a_\mathfrak{p})_{\mathfrak{p}}\).
For every \(\mathfrak{p}\notin \mathfrak{M}\cup S\), we fix a factorization \(a_\mathfrak{p}=\gamma_\mathfrak{p}u_{\mathfrak{p},1}u_{\mathfrak{p},2}\cdots u_{\mathfrak{p},k_{\mathfrak{p}}}\), with \(\gamma_\mathfrak{p}\in \roi_\mathfrak{p}^\times\) and \(u_{\mathfrak{p},i}\in \mathcal{A}(\roi_\mathfrak{p}^\bullet)\).
Then 
\[\varphi(a)=\left(\sideset{}{^\bullet}\prod_{\mathfrak{p}\notin\mathfrak{M}\cup S}[\pi_\mathfrak{p}]_R^{k_\mathfrak{p}}\right)\times \alpha_\mathfrak{M}\beta_\mathfrak{M},\]
and we note that 
\[T_1=\left(\sideset{}{^\bullet}\prod_{\mathfrak{p}\notin \mathfrak{M}\cup S}[\pi_\mathfrak{p}]_R^{d_\mathfrak{p}}\right), d_\mathfrak{p}\leq k_\mathfrak{p}.\]
Let \(\tilde{\alpha}\in\widetilde{\roi}^\bullet\) be defined by
\begin{align*}
\alpha_\mathfrak{p}=
\begin{cases}
\gamma_\mathfrak{p}u_{\mathfrak{p},1}u_{\mathfrak{p},2}\cdots u_{\mathfrak{p},d_\mathfrak{p}}, &\text{if } d_\mathfrak{p}>0,\\
\gamma_\mathfrak{p},&\text{if } d_\mathfrak{p}=0,
\end{cases}
\end{align*}
for all \(\mathfrak{p}\notin \mathfrak{M}\cup S\), and let \(\alpha=(\tilde{\alpha},\alpha_\mathfrak{M})\in \widehat{\roi}^\bullet\).
It is a left divisor of \(a\) in \(\widehat{\roi}^\bullet\) such that \(\phi(\alpha)=(T_1,\alpha_\mathfrak{M})\).
By Lemma \ref{lem:hermite} there exists an \(\epsilon\in \widehat{\roi}^\times\) such that \(a_1=\alpha\epsilon\in \roi^\bullet\).
By construction \(a_1\) is a left divisor of \(a\) in \(\widehat{\roi}^\bullet\) and hence in \(\roi^\bullet\) by Lemma \ref{lem:leftdiv}.
Moreover we have \(\varphi(a_1)=\phi(\alpha)\cdot (1,\epsilon_\mathfrak{M})=(T_1,\alpha_\mathfrak{M}\epsilon_\mathfrak{M})\) as desired. 
It follows that \(\varphi\) is a transfer homomorphism.
\end{proof}

\section{Tiled Orders}
In this section we fix a non archimedean prime \(\mathfrak{p}\) of the number field \(K\), and we let \(A_\mathfrak{p}\cong M_n(D_\mathfrak{p})\) be a central simple algebra over \(K_\mathfrak{p}\) of degree \(nd\), where \(d\) denotes the degree of the \(K_\mathfrak{p}\)-central division algebra \(D_\mathfrak{p}\).
We let \(\Delta_\mathfrak{p}\) denote the unique maximal order of \(D_\mathfrak{p}\) with maximal ideal \(\mathfrak{P}\).
Further we fix uniformizers \(\pi\) and \(\api\) for \(R_\mathfrak{p}\) and \(\Delta_\mathfrak{p}\) respectively.
Finally \(w_\mathfrak{p}\) denotes the valuation on \(D_\mathfrak{p}\).

Let \(\tilde{n}=(n_1,...,n_r)\) be a partition of \(n\) and let \(M=(m_{ij})_{i,j}\in M_r(\Z)\) be a matrix. Using the identification in \eqref{eq:vectordecomp} we define the set
\begin{align}
\roi(\tilde{n},M)=\bigoplus_{i,j=1}^r M_{n_i\times n_j}(\mathfrak{P}^{m_{ij}})\subset M_n(D_\mathfrak{p}),
\end{align}
and note that \(\roi(\tilde{n},M)\) is an order if and only if \(m_{ii}=0\) for all \(i\in\{1,...,r\}\), and \(m_{ij}+m_{jk}\geq m_{ik}\) for all \(i,j,k\in\{1,...,r\}\).
We introduce the notions of tiled local orders (sometimes referred to as graduated orders), and orders in standard form, following \cite{plesken}.

\begin{definition}\label{def:tiledorder}
A local order \(\roi\subset M_n(D_\mathfrak{p})\) is \textit{tiled} if it contains a conjugate of the ring \(\text{Diag}(\Delta_\mathfrak{p},...,\Delta_\mathfrak{p})\).
A tiled order \(\roi\) is said to be in \textit{standard form} if there exists a partition \(\tilde{n}=(n_1,...,n_r)\) of \(n\), and a matrix \(M=(m_{ij})\in M_r(\Z)\), such that \(\roi=\roi(\tilde{n},M)\) and
\begin{enumerate}
\item \(m_{ij}+m_{jk}\geq m_{ik}\),
\item \(m_{ii}=0\),
\item \(m_{ij}+m_{ji}>0\) for all \(i\neq j\).
\end{enumerate}
An order in standard form comes equipped with \textit{structural invariants} \(m_{ijk}=m_{ij}+m_{jk}-m_{ik}\geq 0\) for all triples \(i,j,k\in\{1,...,n\}\).
\end{definition}

Tiled orders may be viewed as the higher dimensional analogue of Eichler orders in \(M_2(K_\mathfrak{p})\).
They have seen some study in recent years \cite{babei2020,babei2019,Shemanske2010}.
If \(n=1\) then \(A_\mathfrak{p}=D_\mathfrak{p}\) is a division ring.
In this case, a local order is tiled if and only if it is the unique maximal order \(\Delta_\mathfrak{p}\).
Let \(\roi=\roi(\tilde{n},M)\) be an order in standard form with \(\tilde{n}=(n_1,...,n_r)\).
The full symmetric group \(S_r\) acts on the coefficients of \(\tilde{n}\), and gives rise to a set of orders in standard form isomorphic to \(\roi\);
For any \(\sigma\in S_r\) we define \(\roi_\sigma=\roi(\tilde{n}',M')\), where \(\tilde{n}'=(n_{\sigma(1)},...,n_{\sigma(r)})\), and \(M'_{ij}=m_{\sigma(i)\sigma(j)}\).
This order will be referred to as the \textit{order derived from} \(\roi\) \textit{under} \(\sigma\). 

\begin{proposition}[{\cite[Proposition II.6]{plesken}}]\label{lem:standardformiso}
Any tiled local order is isomorphic to an order in standard form.
Let \(\roi=\roi(\tilde{n},M)\) and \(\roi(\tilde{n}',M')\) be two orders in standard form with \(\tilde{n}=(n_1,...,n_r)\), \(\tilde{n}'=(n'_1,....,n'_s)\), and structural invariants \(m_{ijk}\), and \(m'_{ijk}\) respectively.
Then \(\roi\) and \(\roi'\) are isomorphic if and only if \(r=s\), and there exists a permutation \(\sigma\in S_r\) such that \(n'_i=n_{\sigma(i)}\) for all \(i\in\{1,...,r\}\), and \(m'_{ijk}=m_{\sigma(i)\sigma(j)\sigma(k)}\) for all \(i,j,k\in\{1,...,r\}\).
\end{proposition}

\begin{proof}
Let \(\roi\) be a tiled order.
After conjugation we may assume that \(\roi\) contains \(\text{Diag}(\Delta_\mathfrak{p},...,\Delta_\mathfrak{p})\), and hence that \(\roi =\roi((1,1,...,1),M)\) for some matrix \(M\in M_n(\Z)\).
Proceeding with induction we may assume that \(\roi=\roi(\tilde{n},M)\) with \(\tilde{n}=(n_1,...,n_{r-1},1)\), and \(m_{ij}+m_{ji}>0\) for all \(i,j\leq r-1\).
If \(m_{ir}+m_{ri}>0\) for all \(i\), we are done.
Hence assume there exists an \(i\) such that \(m_{ir}+m_{ri}=0\).
Conjugating \(\roi\) with a permutation matrix if necessary, we may assume that \(i=r-1\).
Conjugating \(\roi\) from the left  with the diagonal matrix whose \(r\)-th entry is \(\api^{-m_{ri}}\), and \(1\) elsewhere, it follows that \(m_{ir}=m_{ri}=0\), and hence that the obtained order is in standard form.
The second part is proven in \cite[Proposition II.6]{plesken}. 
\end{proof}

A well-known application of orders in standard form is given in \cite[Theorem 39.14]{reinermaximal}, where it is shown that an order \(\roi\subset M_m(D_\mathfrak{p})\) is hereditary, if and only if it is isomorphic to an order \(\roi'=\roi(\tilde{n},M)\) in standard form, such that \(m_{ij}=0\), if \(i\geq j\), and \(m_{ij}=1\), if \(i<j\).
In light of Proposition \ref{lem:standardformiso} we provide another classification of local hereditary orders.

\begin{corollary}\label{cor:standardformhereditary}
Let \(\roi=\roi(\tilde{n},M)\) be an order in standard form.
Then \(\roi\) is hereditary if and only if \(m_{ij}+m_{ji}=1\) for all pairs \(i\neq j\).
\end{corollary}

\begin{proof}
Aassume first that \(\roi\) is hereditary.
By  \cite[Theorem 39.14]{reinermaximal}, \(\roi\) is isomorphic to an order \(\roi'=\roi(n',M')\) such that \(m'_{ij}=0\), if \(i\geq j\), and \(m'_{ij}=1\), if \(i<j\).
Since \(\roi\) and \(\roi'\) are isomorphic, there exists a permutation \(\sigma\in S_r\) such that 
\(m_{ij}+m_{ji}=m_{iji}=m'_{\sigma(i)\sigma(j)\sigma(i)}=m'_{\sigma(i)\sigma(j)}+m'_{\sigma(j)\sigma(i)}=1\), proving the first implication.

Conversely suppose \(m_{ij}+m_{ji}=1\) for all pairs \(i\neq j\).
Proceeding with induction on \(r\), and using the fact that all isomorphisms of \(M_{n-n_r}(D_\mathfrak{p})\) are inner, we may assume that for all \(i,j\leq r-1\) we have \(m_{ij}=0\), if \(i\geq j\), and \(m_{ij}=1\), if \(i<j\).
Moreover after conjugation with an appropriate matrix we may assume that \(m_{r1}=0\) and \(m_{1r}=1\).
Using the first property of Definition \ref{def:tiledorder} we find that \(0\leq m_{r,k}\leq m_{r,k+1}\) for all \(k<r-1\), and \(m_{r,k}\leq m_{1,k}=1\) for all \(k<r\).
Hence there exists a unique \(k\leq r\) such that \(m_{rl}=0\) for all \(l\leq k\) and \(m_{rl}=1\) for all \(k<l<r\).
Let \(\sigma\in S_r\) be the permutation given by 
\[\sigma(i)=\begin{cases}
i, & \text{if } i\leq k,\\
i+1, & \text{if } k+1 \leq i \leq r-1,\\
k+1, & \text{if } i=r,
\end{cases}\]
and let \(\roi_\sigma\) be the order derived from \(\roi\) under \(\sigma\).
Then \(\roi_\sigma\) satisfies the conditions of \cite[Theorem 39.14]{reinermaximal} and is therefore hereditary.
It follows that \(\roi\) is hereditary as well.
\end{proof}

\begin{proposition}\label{prop:invertible}
Let \(\roi=\roi(\tilde{n},M)\) be a tiled order in standard form with \(\tilde{n}=(n_1,...,n_r)\).
Then a matrix \(B=[B_{n_k\times n_l}]\in\roi\) is invertible in \(\roi\) if and only if \(B_{n_k\times n_k}\) is invertible in \(M_{n_k}(\Delta_\mathfrak{p})\) for all \(k\in\{1,...,r\}\).
\end{proposition}
\begin{proof}
Let \(\roi=\roi(\tilde{n},M)\) be a tiled order in standard form.
Let \(B=[B_{n_k\times n_l}]\in \roi\) be a matrix, and write \(\overline{B}=[\overline{B_{n_k\times n_l}}]\) for its image in \(\roi/\text{rad}(\roi)\).
Then \(B\) is invertible in \(\roi\) if and only if \(\overline{B}\) is invertible in \(\roi/\text{rad}(\roi)\).
Since \(\text{rad}(\roi)=\roi(\tilde{n},I_r+M)\) (see \cite[Remark II.4]{plesken}), we find that \(\roi/\text{rad}(\roi)\cong\bigoplus_{k}M_{n_k}(\Delta_\mathfrak{p}/\mathfrak{P})\), and that \(\overline{B}=\bigoplus_k\overline{B_{n_k\times n_k}}\) proving the claim.
\end{proof}

Let \(\roi=\roi(\tilde{n},M)\) be an order in standard form with \(\tilde{n}=(n_1,...,n_r)\). 
For every \(l\in\{1,...,r\}\) we let \(\roi^{(l)}=\roi(\tilde{n}^{(l)},M^{(l)})\subset M_{n_1+...+n_l}(\Delta_\mathfrak{p})\) be the order in standard form given by \(\tilde{n}^{(l)}=(n_1,...,n_l)\) and \(M^{(l)}_{ij}=m_{ij}\) for all \(i,j\in\{1,...,l\}\).
Note that \(\roi^{(r)}=\roi\). 
For any pair \(k,l\in\{1,...,r\}\) with \(l\leq k\), there exists a natural injective monoid homomorphism \(\tau_{kl}:\roi^{(l)\bullet}\hookrightarrow \roi^{(k)\bullet}\) given by extending an element with the identity on the diagonal.
For any triple \(l\leq k\leq s\) we have \(\tau_{sl}=\tau_{sk}\circ \tau_{kl}\).
In general we will write \(\alpha^{(k)}\) for \(\tau_{kl}(\alpha)\) if this is clear from context.

\begin{lemma}\label{lem:atomblockreduction}
Let \(\roi=\roi(\tilde{n},M)\) be an order in standard form with \(\tilde{n}=(n_1,...,n_r)\), let \(l\in \{1,...,r\}\) and let \(\alpha\in \roi^{(l)}\).
Then \(\alpha\) is an atom in \(\roi^{(l)}\) if and only if \(\alpha^{(k)}\) is an atom in \(\roi^{(k)}\) for all \(k\geq l\). 
\end{lemma}

\begin{proof}
If \(\alpha^{(k)}\) is an atom for all \(k\geq l\), then clearly \(\alpha=\alpha^{(l)}\) is an atom.
Assume that \(\alpha\) is an atom in \(\roi^{(l)}\).
By induction it suffices to prove that \(\alpha^{(l+1)}\) is an atom in \(\roi^{(l+1)}\), hence we may assume that \(l=r-1\).
Let \(\alpha\in\roi^{(r-1)}\) be an atom and assume \(\alpha^{(r)}=BC\) for certain \(B,C\in \roi^{(r)}\).
Write \(\alpha^{(r)}=[A_{n_k\times n_s}]\), \(B=[B_{n_k\times n_s}]\) and \(C=[C_{n_k\times n_s}]\).
By definition of \(\alpha^{(r)}\) it follows that 
\[I_{n_r}=B_{n_r\times n_r}C_{n_r\times n_r} +\sum_{i=1}^{r-1}B_{n_r\times n_i}C_{n_i\times n_r}.\]
As \(\roi\) is in standard form, the sum on the right hand side is contained in \(\mathfrak{P}M_{n_r}(\Delta_\mathfrak{p})\), hence \(B_{n_r\times n_r}\) and \(C_{n_r\times n_r}\) are invertible by Lemma \ref{cor:invertnorm}.
Multiplying \(C\) from the left with an appropriate unit we may assume that \(C_{n_r\times n_r}=I_{n_r}\).
Further multplying \(C\) from the left with the matrix \(\Gamma_r=[\Gamma_{n_k\times n_s}]\in \roi^{(r)\times}\), given by \(\Gamma_{n_k\times n_k}=I_{n_k}\) for all \(k\), \(\Gamma_{n_k\times n_r}=-C_{n_k\times n_r}\) for all \(k<r\), and 0 elsewhere, we may also assume that \(C_{n_k\times n_r}=0\) for all \(k<r\).
Hence we may assume that \(B_{n_k\times n_r}=A_{n_k\times n_r}=0\) if \(k<r\).
It follows that
\[\alpha^{(r)}=\left(\begin{array}{c|c}
\alpha & 0\\ \hline
0 & I_{n_r}
\end{array}\right)
=\left(\begin{array}{c|c}
\beta & 0\\ \hline
X & I_{n_r}
\end{array}\right)\cdot\left(\begin{array}{c|c}
\gamma & 0\\ \hline
Y & I_{n_r}
\end{array}\right)=BC,\]
for certain \(\beta,\gamma\in \roi^{(r-1)}\), and matrices \(X,Y\in M_{n_r\times n-n_r}(D_\mathfrak{p})\).
Consequently we find that \(\alpha=\beta\gamma\), and 
since \(\alpha\) is an atom in \(\roi^{(r-1)}\), it follows that \(\beta\) or \(\gamma\) must be invertible in \(\roi^{(l)}\).
Hence by Proposition \ref{prop:invertible} either \(B\) or \(C\) must be invertible in \(\roi^{(r)}\).
It follows that \(\alpha^{(r)}\) is an atom.
\end{proof}

\begin{proposition}\label{prop:standardelas}
Let \(\roi=\roi(\tilde{n},M)\subset M_n(D_\mathfrak{p})\) be an order in standard form.
The following are equivalent
\begin{enumerate}
\item \(\roi\) has finite elasticity,
\item \(m_{ij}+m_{ji}=1\) for all \(i\neq j\in \{1,...,r\}\),
\item \(\roi\) is hereditary.
\end{enumerate}
Moreover if any of the above statements fail to hold, there exists a set of atoms \(\{\alpha_k,\alpha'_k\mid k\in \Z_{\geq 1}\}\) of \(\roi\) satisfying
\[\alpha_k\alpha'_k=\alpha_1^k\alpha'^k_1\]
for all \(k\).
In particular \(\rho_k(\roi^\bullet)=\infty\) for all \(k\geq 2\).
And for all \(k\in \Z_{\geq 1}\) there are inclusions \(2\cdot\Z_{\geq 1}\subset \mathcal{U}_{2k}(\roi^\bullet)\), and \(1+2\cdot\Z_{\geq 1}\subset \mathcal{U}_{2k+1}(\roi^\bullet)\).
\end{proposition}

\begin{proof}
The implications \(2\Leftrightarrow 3\) are shown in Corollary \ref{cor:standardformhereditary}.
Furthermore, the implication \(3\Rightarrow 1\) follows from Lemma \ref{lem:hereditarytransfer}.
Thus it remains to show that \(1\Rightarrow 2\).
By Lemma \ref{lem:atomblockreduction} it suffices to prove that the statement holds for \(r=2\),
hence assume that \(\tilde{n}=(n_1,n_2)\) and that \(m_{12}+m_{21}=t>1\).
Without loss of generality we may assume that \(m_{21}=0\), and that \(m_{12}=t\).
We aim to construct the set of atoms as described in the proposition, and let
\begin{align*}
\alpha_k=\left(\begin{array}{cc|cc}
I_{n_1-1} & 0 & 0 & 0\\
0 & \api^{k-1}+\api^{t-1} &\api^t &0\\ \hline
0 & 1 & \api & 0\\
0 & 0 & 0 & I_{n_2-1}
\end{array}\right), && \alpha'_k= \left(\begin{array}{cc|cc}
I_{n_1-1} & 0 & 0 & 0\\
0 & \api &-\api^t &0\\ \hline
0 & -1 & \api^{t-1}+\api^{k-1} & 0\\
0 & 0 & 0 & I_{n_2-1}
\end{array}\right).
\end{align*}
Assume that \(\alpha_k= BC\) for certain matrices \(B,C\in \roi\), and write
\begin{align*}
\alpha_k=\left(\begin{array}{c|c}
A_{n_1\times n_1} & \api^t A_{n_1\times n_2} \\ \hline
A_{n_2\times n_1} & A_{n_2\times n_2}
\end{array}\right)
=\left(\begin{array}{c|c}
B_{n_1\times n_1} & \api^tB_{n_1\times n_2} \\ \hline
B_{n_2\times n_1} & B_{n_2\times n_2}
\end{array}\right)
\cdot \left(\begin{array}{c|c}
C_{n_1\times n_1} &  C_{n_1\times n_2} \api^t \\ \hline
C_{n_2\times n_1} & C_{n_2\times n_2}
\end{array}\right)
\end{align*}
Applying Lemma \ref{cor:invertnorm} to \(A_{n_2\times n_2}= B_{n_2\times n_2}C_{n_2\times n_2} +\api^t D\), we find that there exists a \(\gamma\in M_{n_2}(\Delta_\mathfrak{p})^\times\) such that \(B_{n_2\times n_2}C_{n_2\times n_2}=\gamma A_{n_2\times n_2}\).
Since \(A_{n_2\times n_2}\) is an atom in \(M_{n_2}(\Delta_\mathfrak{p})\) it follows that either \(B_{n_2\times n_2}\) or \(C_{n_2\times n_2}\) is invertible.
Let us first assume that \(C_{n_2\times n_2}\) is invertible.
We may then assume it is the identity, and using row reduction, we may also assume \(C_{n_1\times n_2}=0\).
Consequently we find that \(B_{n_i\times n_2}=A_{n_i\times n_2}\) for \(i=1,2\).
Using standard matrix notation write \(\alpha_k=(a_{i,j})_{i,j}\), \(B=(b_{i,j})_{i,j}\), and \(C=(c_{i,j})_{i,j}\). 
The equality
\[1=a_{n_1+1,n_1} =\sum_{i=1}^{n_1} b_{n_1+1,i}c_{i,n_1}+ \api c_{n_1+1,n_1}\]
shows that there exists an \(i\leq n_1\) such that \(c_{i,n_1}\) is a unit.
Multiplying \(C\) from the left with an appropriate unit we may assume that \(i=n_1\).
Proceeding with row reduction we may now assume that \(c_{l,n_1}=\delta_{l,n_1}\), where \(\delta_{i,j}\) is the Kronecker delta.
It therefore follows that \(b_{i,n_1}=a_{i,n_1}\) for \(i=1,...,n\).
But as \(\alpha_k\) is block diagonal, it now follows from Lemma \ref{lem:block diagonal} that \(\alpha_k\) is a right divisor of \(B\) in \(M_n(\Delta_\mathfrak{p})\).
As \(\alpha_k=BC\) it therefore follows that \(C\) is invertible.

Analogously we now assume that \(B_{n_2\times n_2}\) is invertible.
We may further assume it to be the identity and, using column reduction, we may assume that \(B_{n_2\times n_1}=0\).
Consequently we find that \(C_{n_2\times n_i}=A_{n_2\times n_i}\).
Looking now at \(a_{n_1,n_1+1}\) we see that
\[\api^t =\sum_{i=1}^{n_1}b_{n_1,i} c_{i,n_1+1}\api^t+\api^{t} b_{n_1,n_1+1}\api,\]
hence there exists a \(k\leq n_1\) such that \(b_{mn_1,k}\) is a unit.
Multiplying \(B\) from the right with an appropriate unit we may assume \(k=n_1\) and, using column reduction, we may further assume \(b_{n_1,l}=\delta_{n_1,l}\).
Analogous to the previous case we find that \(c_{n_1,l}=a_{n_1,l}\) and using the analogous statement of Lemma \ref{lem:block diagonal} for left divisors we conclude that \(\alpha_k\) is a left divisor of \(C\).
It follows that \(\alpha_k\) is an atom in \(\roi\).

Analogously it follows that \(\alpha'_k\) is an atom in \(\roi\), and
an easy computation shows that \(\alpha_k\alpha'_k=\alpha_1^k\alpha'^k_1\) as desired.
The remaining statements follow directly.
\end{proof}

\begin{corollary}\label{cor:tiledlocalelas}
Let \(\roi\) be a tiled local order.
Then \(\roi\) has finite elasticity if and only if it is hereditary.
\end{corollary}

\section{Transfer Results}
In this section we  fix a central simple algebra \(A/K\) of degree \(n\).
Further we let \(\text{TRam}(A)\subset \text{Pl}(K)\) denote the finite set of places of \(K\) where \(A_\mathfrak{p}\) is a division ring.

\begin{definition}
An \(R_S\)-order \(\roi\subset A\) is \textit{tiled away from} \(\text{TRam}(A)\), if \(\roi_\mathfrak{p}\) is a tiled order for all primes \(\mathfrak{p}\notin \text{TRam}(A)\cup S\).
The order \(\roi\) is \textit{locally tiled} if \(\roi_\mathfrak{p}\) is a tiled order for all non-zero primes \(\mathfrak{p}\) of \(R_S\).
\end{definition}

Theorem \ref{prop:transfer} allows us to express the elasticity of \(\roi\) in terms of its local elasticities when \(\roi\) is  quaternion or tiled away from \(\text{TRam}(A)\).
Before proving this we recall some facts about the local elasticities of quaternion orders and orders tiled away from \(\text{TRam}(A)\).

\begin{proposition}\label{prop:elasprop}
Let \(\roi\subset A\) be an \(R_S\)-order.
Assume that \(n=2\) or that \(\roi\) is tiled away from \(\text{TRam}(A)\).
Let \(\mathfrak{M}\) denote the set of non-zero primes of \(R_S\) where \(\roi_\mathfrak{p}\) is not hereditary.
The following statements hold:
\begin{enumerate}
\item \(\rho(\roi_\mathfrak{p}^\bullet)=1\) for all \(\mathfrak{p}\notin \mathfrak{M}\),
\item \(\rho(\roi_\mathfrak{p}^\bullet)< \infty\) for all \(\mathfrak{p}\in \text{TRam}(A)\),
\item \(\rho(\roi_\mathfrak{p}^\bullet)=\infty \) for all \(\mathfrak{p}\in \mathfrak{M}\backslash \text{TRam}(A)\).
\end{enumerate}
\end{proposition}
\begin{proof}
If \(\mathfrak{p}\notin \mathfrak{M}\) then \(\roi_\mathfrak{p}\) is hereditary hence the claim follows from Lemma \ref{lem:hereditarytransfer}.
If \(\mathfrak{p}\in \text{TRam}(A)\) the result follows from \cite[Theorem 3.1]{baeth2017}.
The third statement follows from Proposition \ref{prop:standardelas} and \cite[Theorem 5.8]{baeth2017}.
\end{proof}

\begin{proposition}\label{lem:elasdescentgeneral}
Let \(\roi\subset A\) be a Hermite \(R_S\)-order.
Let \(\mathfrak{M}\) denote the finite set of non-zero primes of \(R_S\) where \(\roi\) is not hereditary. 
Assume there exists a prime \(\mathfrak{p}\in\mathfrak{M}\) together with an integer \(d\in\Z_{\geq 2}\), an infinite set \(I\subset \N\), and a system of atoms \(\{\alpha_{k,1},\alpha_{k,2},...,\alpha_{k,d}\mid k\in I\}\subset \mathcal{A}(\roi_\mathfrak{p}^\bullet)\) such that for all \(k\in I\), there exists a non-unit \(\gamma_k \in \roi_\mathfrak{p}^\bullet\) satisfying
\[\alpha_{k,1}\alpha_{k,2}\cdots\alpha_{k,d}=\beta_{k}\gamma_k^k\delta_k,\]
for some \(\beta_k,\delta_k\in\roi_\mathfrak{p}^\bullet\).
Then \(\rho_{d+1}(\roi^\bullet)=\infty\).
Furthermore, if for all \(k\in I\) we have \(\beta_k\in \roi^\bullet\) or \(\delta_k\in \roi^\bullet\), then \(\rho_d(\roi^\bullet)=\infty\).
\end{proposition}

\begin{proof}
Let \(\mathcal{B}\) and \(\mathcal{C}\) denote the \(T\)-block monoid and ray class group associated to \(\roi\) respectively. Write \(C=\#\mathcal{C}\)  and let \(\mathcal{P}=\{\pi_\mathfrak{q}\mid \mathfrak{q}\notin \mathfrak{M}\cup S\}\) be a set of uniformizers for the maximal ideals of \(R_\mathfrak{q}\) for all \(\mathfrak{q}\notin \mathfrak{M}\cup S\).
Let \(m\in \N\).
Let \(k\in I\) be an integer satisfying \(k\geq mC\), and let \(\beta_k,\gamma_k,
\delta_k\) be as given in the proposition.
If \([\nrd \alpha_{k,i}]_R=[1]_R\in \mathcal{C}\), then \((1,\alpha_{k,i})\in \mathcal{B}\) is an atom.
For all \(i\leq d\) such that \([\nrd \alpha_{k,i}]_R\neq [1]_R\), we let \(\pi_i\in \mathcal{P}\) be a uniformizer  such that \(([\pi_i]_R,\alpha_{k,i})\in \mathcal{B}\), and note that \(([\pi_i]_R,\alpha_{k,i})\in \mathcal{A}(\mathcal{B})\) as well.
Without loss of generality we hence impose that \([\nrd \alpha_{k,i}]_R\neq [1]_R\) for all \(i\).
Assume further that \([\nrd\beta_k]_R\neq [1]_R\), and let \(\omega,\omega'\in \mathcal{P}\) be such that \(([\omega]_R,\beta_k),([\omega]_R\bigcdot[\omega']_R,1)\in \mathcal{B}\).
The latter is an atom in \(\mathcal{B}\), and we have a factorization 
\begin{align*}
([\omega]_R\bigcdot[\omega']_R,1)([\pi_1]_R,\alpha_{k,1})\cdots([\pi_d]_R,\alpha'_{k,d})
=([\omega]_R,\beta_k)\left(1,\gamma_k^C\right)^m \left([\omega']_R\bigcdot\sideset{}{^\bullet}\prod_i[\pi_i]_R,\gamma_k^{k-mC}\delta_k\right).
\end{align*}
Hence \(\rho_{d+1}(\roi^\bullet)=\rho_{d+1}(\mathcal{B})\geq m\).
If \([\nrd\beta_k]_R= [1]_R\), then the factor \(([\omega]_R\bigcdot[\omega']_R,1)\) may be omitted.
Hence we find \(\rho_d(\roi^\bullet)\geq m\).
By symmetry the same holds if \([\nrd \delta_k]_R\neq [1]_R\) and the claim follows.
\end{proof}

Note that we may replace the factor \(\gamma_k^k\) in Proposition \ref{lem:elasdescentgeneral} with a product of cancellative non-units \(\prod_{i=1}^{k}\gamma_{k,i}\) such that \([\nrd \gamma_{k,i}]_R=[\nrd \gamma_{k,j}]_R\) for all \(i,j\in\{1,...,k\}\).

\begin{corollary}\label{lem:elasdescent}
Let \(\roi\subset A\) be a Hermite \(R_S\)-order and assume that \(n=2\) or that \(\roi\) is tiled away from \(\text{TRam}(A)\).
If there exists a non-zero prime \(\mathfrak{p}\) of \(R_S\) such that \(\rho(\roi_\mathfrak{p}^\bullet)=\infty\), then \(\rho_2(\roi^\bullet)=\infty\). 
\end{corollary}

\begin{proof}
By Proposition \ref{prop:elasprop} we must have \(\mathfrak{p}\notin \text{TRam}(A)\).
If \(\roi_\mathfrak{p}\) is tiled we may assume it to be in standard form.
In this case, the set described in Proposition \ref{prop:standardelas} satisfies the conditions of Proposition \ref{lem:elasdescentgeneral}.
Otherwise \(\roi_\mathfrak{p}\) must be a non-Eichler order inside \(M_2(K_\mathfrak{p})\).
By \cite[Theorem 5.8]{baeth2017} there exists an integer \(N\) such that for every \(k\geq N\) there exist atoms \(\alpha_k,\alpha'_k\) satisfying \(\alpha_k\alpha'_k=\mu_k\pi^{2k}\), for some unit \(\mu_k\in R_\mathfrak{p}^\times\).
Hence the claim follows in this case as well.
\end{proof}

Let \(\roi\subset A\) be an \(R_S\)-order and let \(\mathfrak{M}\) denote the finite set of primes where \(\roi_\mathfrak{p}\) is not hereditary.
The local valuations \(v_\mathfrak{p}:R_\mathfrak{p}^\bullet\to \Z_{\geq 0}\) for \(\mathfrak{p}\notin S\) allows us to define  monoid homomorphisms \(\widetilde{v},v_\mathfrak{M}:R_S^\bullet \to \Z_{\geq 0}\) by setting
\begin{align}
\widetilde{v}(\alpha)=\sum_{\mathfrak{p}\notin  \mathfrak{M}\cup S}v_\mathfrak{p}(\alpha), && v_\mathfrak{M}(\alpha)=\sum_{\mathfrak{p}\in \mathfrak{M}}v_\mathfrak{p}(\alpha),\label{eq:totvaluation}
\end{align}
for \(\alpha\in R_S^\bullet\).
If \(\mathfrak{M}=\{\mathfrak{p}\}\) is the collection of a single prime then \(v_\mathfrak{M}=v_\mathfrak{p}\).
We have \(\alpha\in R_S^\times\) if and only if \(\widetilde{v}(\alpha)+v_\mathfrak{M}(\alpha)=0\), and we have \(\alpha\in \mathcal{A}(R_S^\bullet)\) if \(\widetilde{v}(\alpha)+v_\mathfrak{M}(\alpha)=1\).
Furthermore, since the reduced norm \(\roi^\bullet\to R_S^\bullet\) is a monoid homomorphism we find that 
\(\alpha\in \roi^\times\) if and only if \(\widetilde{v}(\nrd \alpha)+v_\mathfrak{M}(\nrd\alpha)=0\), and that \(\alpha\in \mathcal{A}(\roi^\bullet)\) if \(\widetilde{v}(\nrd \alpha)+v_\mathfrak{M}(\nrd\alpha)=1\).

\begin{theorem}\label{thm:elasequivalence}
Let \(\roi\subset A\) be a Hermite \(R_S\)-order and assume that \(n=2\) or that \(\roi\) is tiled away from \(\text{TRam}(A)\).
The following are equivalent:
\begin{enumerate}
\item \(\rho(\roi^\bullet)<\infty\),
\item \(\rho_k(\roi^\bullet)<\infty\) for all \(k\geq 2\),
\item \(\rho_2(\roi^\bullet)<\infty\),
\item \(\rho(\roi^\bullet_\mathfrak{p})<\infty\) for all \(\mathfrak{p}\notin S\),
\item \(\roi_\mathfrak{p}\) is hereditary for all \(\mathfrak{p}\notin \text{TRam}(A)\).
\end{enumerate}
Furthermore, let \(\roi'\subset A\) be a hereditary order containing \(\roi\).
Let \(\mathfrak{M}\subset \text{TRam}(A)\) denote the set of primes of \(R_S\) where \(\roi_\mathfrak{p}\) is not hereditary, and let \(\mathcal{C}\) be the ray class group of \(K\) associated to \(\roi\).
If any of the above statements hold, we have
\[\rho(\roi^\bullet)\leq \begin{cases}
\max\{1,\frac{D(\mathcal{C})}{2}\},& \text{if } \roi \text{ is hereditary,}\\
2n\cdot\#\mathfrak{M}\cdot\left(v_\mathfrak{M}([\roi':\roi])+D(\mathcal{C})\right),& \text{if } \roi \text{ is not hereditary}.
\end{cases}\]
Here \(D(\mathcal{C})\) is the Davenport constant of \(\mathcal{C}\).
\end{theorem}

\begin{proof}
The implication \(1\Rightarrow 2\) follows directly from Lemma \ref{lem:elaslimit}.
The implication \(2\Rightarrow 3\) is immediate, and the implication \(3\Rightarrow 4\) is proven in Corollary \ref{lem:elasdescent}.
Furthermore the implications \(4\Leftrightarrow 5\) follow from Proposition \ref{prop:elasprop}.
Thus it remains to show that \(5 \Rightarrow 1\).

Let \(\mathcal{B}\) be the \(T\)-block monoid associated to \(\roi\), and let \(\varphi:\roi^\bullet\to \mathcal{B}\) be the transfer homomorphism given in Theorem \ref{prop:transfer}.
Let us first assume that \(\roi\) is hereditary.
Then \(\varphi\) is a transfer homomorphism, hence \(\rho(\roi^\bullet)=\rho(\mathcal{B})=\max\{1,\frac{\mathcal{D}(\mathcal{C})}{2}\}\) by \cite[Theorem 2.3.1]{Geroldinger09}.
Hence assume that \(\roi\) is not hereditary.
For all primes \(\mathfrak{p}\in\mathfrak{M}\), let \(\pi_\mathfrak{p}\) be a uniformizer for \(R_\mathfrak{p}\), and write \(\Delta_\mathfrak{p}\) for the unique maximal order in the division ring \(A_\mathfrak{p}\).
Then there exists a unique minimal \(\kappa_\mathfrak{p}\in \Z_{\geq 1}\) such that \(\pi_\mathfrak{p}^{\kappa_\mathfrak{p}}\Delta_\mathfrak{p}\subset \roi_\mathfrak{p}\).
Since \(\roi'\) is hereditary we have \(\roi'_\mathfrak{p}=\Delta_\mathfrak{p}\) by \cite[Lemma 21.2.7]{Voightquatalg}.
Hence we have \(\kappa_\mathfrak{p}=v_\mathfrak{p}([\roi':\roi])\), and we write \(\kappa=v_\mathfrak{M}([\roi':\roi])=\sum_{\mathfrak{p}\in\mathfrak{M}}\kappa_\mathfrak{p}\).

Identify \(\widehat{\roi}=\widetilde{\roi}\times \roi_\mathfrak{M}\), and write \(M=\#\mathfrak{M}\).
Let \(\mathfrak{p}\notin \mathfrak{M}\) be a non-zero prime of \(R_S\).
Since \(\roi_\mathfrak{p}\) is hereditary, every element of \(\roi_\mathfrak{p}^\bullet\) has unique length, and every atom has the some norm up to multiplication by units.
In particular we have \(L(\alpha)=\{v_\mathfrak{p}(\nrd \alpha)\}\).
Fence for every \(\tilde{\alpha}\in \widetilde{\roi}^\bullet\) we have \(L(\tilde{\alpha})=\{\widetilde{v}(\nrd \tilde{\alpha})\}\).
Let \(a\) be any atom of \(\roi^\bullet\).
Write \(\varphi(a)=(T,a_\mathfrak{M})\) and note that
\(T\) is a sequence of length \(\widetilde{v}(\nrd a)\).
Since \(a\) is an atom we must have \(\widetilde{v}(\nrd a)\leq D(\mathcal{C})\), else \(T\) contains a zero-sum subsequence.
Assume that \(\roi^\bullet\) has infinite elasticity.
Let \(m\in \N\) and assume that \(m\geq 2D(\mathcal{C})\).
By assumption there exists an \(a_m\in \roi^\bullet\) such that 
\(m\leq \rho(a_m)=\frac{N}{k}\), where \(N\) and \(k\) are the maximal and minimal lengths respectively.
We let \(a_m=\beta_1\cdots\beta_N\) be a factorization of maximal length and \(a_m=\alpha_1\cdots \alpha_k\) be a factorization of minimal length.
We claim that there are at least \(\frac{N}{2}\) atoms \(\beta_i\) for which \(\widetilde{v}(\nrd \beta_i)=0\).
This is trivially true if \(\widetilde{v}(\nrd a_m)=0\), hence assume \(\widetilde{v}(\nrd a_m)\neq 0\).
By the above we have \(\widetilde{v}(\nrd a_m) \leq kD(\mathcal{C})\), hence
\[2D(\mathcal{C}) \leq m\leq \frac{N}{k}\leq \frac{ND(\mathcal{C})}{\widetilde{v}(\nrd a_m)},\]
which shows that \(\widetilde{v}(\nrd a_m)\leq \frac{N}{2}\).
If \(l\leq N\) denotes the number of atoms \(\beta_i\) for which \(\widetilde{v}(\nrd \beta_i)\neq 0\), then we have \(l\leq \widetilde{v}(\nrd a_m)\leq \frac{N}{2}\) proving the claim.
Equivalently we find that there are at least \(\frac{N}{2}\) atoms \(\beta_i\) for which \(v_\mathfrak{M}(\nrd \beta_i)\neq 0\).
It follows that there exists a prime \(\mathfrak{p}\in \mathfrak{M}\) such that \(v_\mathfrak{p}(\nrd \beta_i)\neq 0\) for at least \(\frac{N}{2M}\) atoms.
Since \(N\geq km\) we conclude that \(v_\mathfrak{p}(\nrd a_m)\geq \frac{km}{2M}\).
In particular there exists an atom \(\alpha_i\) such that \(v_\mathfrak{p}(\nrd \alpha_i)\geq \frac{m}{2M}\).

Now assume that \(m\geq 2nM(\kappa+D(\mathcal{C}))\).
Let \(a,b\in\roi_\mathfrak{p}^\bullet\) be two elements such that \(v_\mathfrak{p}(\nrd b) +n\kappa_\mathfrak{p} \leq v_\mathfrak{p}(\nrd a)\).
Since \(\Delta_\mathfrak{p}\) is the valuation ring of \(A_\mathfrak{p}\), and since \(v_\mathfrak{p}(\nrd b^{-1}a)\geq n\kappa_\mathfrak{p}\) we conclude that \(b^{-1}a\in \pi_\mathfrak{p}^{\kappa_\mathfrak{p}}\Delta_\mathfrak{p}\subset \roi_\mathfrak{p}\), and thus \(a=bc\) for some \(c\in \roi_\mathfrak{p}^\bullet\). 
Since \(\pi_\mathfrak{p}\in \roi_\mathfrak{p}^\bullet\), there exists an atom \(\gamma\in \mathcal{A}(\roi_\mathfrak{p}^\bullet)\) such that \(v_\mathfrak{p}(\nrd \gamma)=w_\mathfrak{p}(\gamma)\in \{1,...,n\}\).
write \(C=\text{ord}_\mathcal{C}([\nrd \gamma]_R)\) for the order of \([\nrd \gamma]_R\) in the class group \(\mathcal{C}\).
Then \(([\nrd \gamma]_R)^C\) is a zero-sum sequence in \(\mathcal{B}(\mathcal{C})\) hence \(C\leq D(\mathcal{C})\).
By assumption we have \(v_\mathfrak{p}( \nrd \alpha_i)\geq \frac{m}{2M}\geq n\kappa+nD(\mathcal{C})\geq n\kappa_\mathfrak{p}+nC\).
Thus by the above argument, \(\gamma^C\) is a left divisor of \(\alpha_{i}\) in \(\roi_\mathfrak{p}\).
Consequently \((1,\gamma^C)\) is a left divisor of \(\varphi(\alpha_i)\) in \(\mathcal{B}\), contradicting the atomicity of \(\alpha_i\).
It follows that \(\roi^\bullet\) has finite elasticity.
\end{proof}

\begin{corollary}
Let \(\roi\subset A\) be a Hermite \(R_S\)-order, and assume that \(\roi\) is locally tiled. 
The following are equivalent:
\begin{enumerate}
\item \(\rho(\roi^\bullet)<\infty\),
\item \(\roi\) is hereditary.
\end{enumerate}
\end{corollary}
\begin{proof}
The implication \(2\Rightarrow 1\) is immediate.
Hence assume that \(\rho(\roi^\bullet)<\infty\).
As \(\roi\) is locally tiled, it is tiled away from \(\text{TRam}(A)\).
Hence \(\roi_\mathfrak{p}\) is hereditary for all \(\mathfrak{p}\notin \text{TRam}(A)\) by Theorem \ref{thm:elasequivalence}.
Furthermore  \(\roi_\mathfrak{p}\) is the unique maximal order in \(A_\mathfrak{p}\) for all \(\mathfrak{p}\) in \(\text{TRam}(A)\).
It follows that \(\roi_\mathfrak{p}\) is hereditary for these primes as well.
As \(\roi_\mathfrak{p}\) is hereditary for all primes \(\mathfrak{p}\in \text{Pl}(K)\backslash S\) we conclude that \(\roi\) is hereditary.
\end{proof}

\printbibliography

\end{document}